\documentclass[leqno]{amsart}
\usepackage{amsmath, amsthm, amssymb, amstext}
\usepackage{scalerel,stackengine}
\usepackage{mathrsfs}
\usepackage{float}
\usepackage{leftindex}
\usepackage{tabularx}

\usepackage[left=2.5cm,right=2.5cm,top=1.5cm,bottom=1.5cm,includeheadfoot]{geometry}
\usepackage{hyperref,xcolor}
\hypersetup{
 pdfborder={0 0 0},
 colorlinks,
}
\usepackage{color}

\usepackage{todonotes}
\usepackage{enumitem}
\allowdisplaybreaks

\newtheorem{theorem}{Theorem}
\newtheorem{remark}[theorem]{Remark}
\newtheorem{lemma}[theorem]{Lemma}
\newtheorem{proposition}[theorem]{Proposition}
\newtheorem{corollary}[theorem]{Corollary}
\newtheorem{definition}[theorem]{Definition}

\numberwithin{theorem}{section} \numberwithin{equation}{section}

\title[Lie symmetry analysis for fractional equation]{Lie symmetry analysis for fractional evolution equation with $\psi$-Riemann-Liouville derivative}

\author[Junior C. A. Soares, Felix S. Costa, J. Vanterler C. Sousa, Maria V. de S. e Sousa \& Amália R. E. Pereira ]{Junior C. A. Soares, Felix S. Costa, J. Vanterler C. Sousa, Maria V.S. Sousa \& Amália R. E. Pereira }

\address[Junior C. A. Soares  ]{Department mathematics, University of Mato Grosso State , Rua A, s/n, 78390-000, Barra do Bugres, Brazil\\ Aerospace Engineering, PPGEA-UEMA, Department of Mathematics, DEMATI-UEMA, 
São Luís, MA 65054, Brazil.}
\email{\tt juniorcasoares@unemat.br}

\address[Felix S. Costa]{Aerospace Engineering, PPGEA-UEMA, Department of Mathematics, DEMATI-UEMA, 
São Luís, MA 65054, Brazil.}
\email{\tt felix@cecen.uema.br}

\address[J. Vanterler da C. Sousa ]{Aerospace Engineering, PPGEA-UEMA, Department of Mathematics, DEMATI-UEMA, 
São Luís, MA 65054, Brazil.}
\email{\tt vanterler@ime.unicamp.br}

\address[Maria V.S. Sousa]{Department of Mathematics, DEMATI-UEMA, 
São Luís, MA 65054, Brazil.}
\email{\tt mariavictoriadesousa03@gmail.com }

\address[A. R. E. Pereira]
{
\newline\indent
Department of Mathematics, DEMATI-UEMA
\newline\indent
São Luís, MA 65054, Brazil.}
\email{\tt{reginaamalia679@gmail.com}}


\begin{document}
 \begin{abstract} We present the applycation of theory of Lie group analysis with $\psi$-Riemann-Liouville fractional derivative detailing the construction of infinitesimal prolongation to obtain Lie symmetries. In additional, is addressed the invariance condition without the need to impose that the lower limit of fractional integral is fixed. We find an expression that expands the knowledge regarding the study of exact solutions for fractional differential equations. We use of the framework developed in \cite{zaky2022note} to present our understanding of the extension of $\psi$-Riemann-Liouville fractional derivative. It is demonstrate the Leibniz type rule for the derivative operator in question for built the prolongation. At last, we calculate the Lie symmetries of the generalized Burgers equation and fractional porous medium equation.

 
\end{abstract}
\maketitle


\section{Introduction}
In the last two decades, fractional calculus has proven to be a very useful theory for modeling physical, biological, financial phenomena and situations that arise in the context of engineering \cite{teodoro2019review,hilfer2000applications,kilbas2006theory,machado2011recent,balthazar2010models}.

The genesis of the study of this area dates back to the year 1695 when, according to the most diverse references in the literature \cite{machado2011recent,hilfer2000applications,de2014review}, 
the correspondence between L’Hospital and Leibniz “rises” the calculus fractional calculus, i.e., a questioning on the half order derivative of a function. The answer to this question provided several insights for the development of concepts and enabled the structuring of the foundations of the research that followed. From 1970, after the congress that took place in the city of New Haven, a beginning was given for the flowering more consistent of applications of the theory of calculus of arbitrary order \cite{kilbas1993fractional}.

Nowadays, classical fractional derivatives are considered to be fractional derivatives in the Riemann-Liouville sense, the Caputo derivative and the Grünwald-Letnikov derivative. As fractional derivatives were widely disseminated via applications in the most diverse areas, the need arose to build new operators that could meet the specific needs of some application contexts.

There are a multitude of definitions of operators, which are considered by their respective authors as fractional derivatives \cite{teodoro2019review}, but it is worth mentioning that, in 2015, Ortigueira \& Machado \cite{ortigueira2015fractional} presented criteria under which a fractional operator should satisfy to be reputed as a fractional derivative. In addition, \cite{tarasov2013no} authors set the condition for a fractional derivative definition. In view of these works, we can list in a preliminary way that fractional operators can be classified \cite{teodoro2017derivadas} as classical derivatives, local derivatives, derivatives with singular kernel and derivatives with non-singular kernel.

In the paper \cite{sousa2018psi}, Sousa and Oliveira inspired by the definition proposed by Samko et al. \cite{kilbas1993fractional}, Almeida \cite{almeida2017caputo} and motivated by the definition of Hilfer fractional derivative \cite{hilfer2009operational}, introduced the $\psi$-Hilfer fractional derivative, which recovers
many particular cases of fractional derivatives \cite{sales},
among them, those derived from Riemann-Liouville, Caputo, Weyl, Chen and Jumarie, $\psi$-Caputo, $\psi$-Riemann-Liouville, Katugampola, Hadamard, Caputo-Hadamard, Caputo Katugampola, Hilfer-Hadamard, Hilfer-Katugampola, Riemann, Prabhakar, Erdélyi-Kober, Liouville, Liouville-Caputo, Riesz, Feller, Cossar and Caputo-Riesz. Another paper that was proposed by Sousa and Oliveira \cite{sousa2019leibniz}, present a generalization of Leibniz rule that met the generic structure of the $\psi$-Hilfer fractional derivative. It is known that Leibniz rule gives a differentiation operator conditions to deal with the derivative of the product of two functions, which frequently emerges in the applications of integer and fractional order differential operators. In particular, we highlight the use of Leibniz rule to obtain Lie symmetries for differential equations involving fractional derivatives.

Furthermore, it is worth mentioning that the 
Lie group theory developed in the $19$th century by the Norwegian mathematician Sophus Lie is a tool that gave conditions for the study of solutions of differential equations via groups of transformations acting under the manifold in which the differential equation is well defined \cite{bluman2010applications,olver2014introduction}. There is a very solid and widespread use of Lie symmetries for the study of differential equations and systems of integer differential equations \cite{oliveri2010lie}. One of the first papers that connected the use of the theory of Lie Symmetries with the derivative of arbitrary order was the work \cite{buckwar1998invariance} in which there was the characterization of exact solutions from symmetries of a differential equation and the glimpse of expansion of the application of Lie symmetries in study of solutions of fractional differential equations.

In \cite{gazizov2007continuous} the authors explicitly presented the infinitesimal extension for differential equations involving the Riemann-Liouville and Caputo fractional derivatives, through the generalized Leibniz rule, and also presented examples. From these works, many others applied the algorithm developed there to find Lie symmetries of differential equations and consequently made use of them, finding order reductions for fractional differential equations.

We can cite as foundation to Lie theory and fractional calculus, the article of  Gazizov et.al \cite{gazizov2012fractional} which presents an explicit  extension formula to the $\psi$-Riemann-Liouville fractional derivative, i.e, derivative of a function with respect to another function with $0<\alpha<1$. Given us the possibility to find non-local symmetries for fractional differential equations. We highlight also the article \cite{costafs}, in which the authors find the explicit infinitesimal prolongation to $\psi$-Hilfer fractional derivative, that is more general case to fractional derivative non-singular kernel cited previously. In order to extend the theoretical framework for the use of Lie symmetries in the case of differential equations of fractional calculus, the authors in \cite{leo2017foundational} presents a detailed approach in what was established as the most appropriate space for analysis of Lie symmetries in the case of fractional differential equations is immersed in an infinite-dimensional jet space. Therefore, the authors used geometric notions to prove a generalized prolongation formula in that space.

In 2020, Zhang \cite{zhang2020symmetry} considered a time-fractional PDE involving Riemann–Liouville fractional partial derivative
given by
\begin{eqnarray}\label{89}
    \partial^{\alpha}_{t} u= E(x,t,u,u_{1},u_{2},u_{3},....,u_{l})
\end{eqnarray}
where $u=u(x,t)$ is the dependent variable of two independent variables $x, t,$ and  $u_0 = u$, $\partial^{\alpha}_{t}(\cdot)$ is the Riemann–Liouville fractional partial derivative of $u=u(x,t)$ of order $0<\alpha<1$. In this sense, the author presented a simpler form for the infinitesimal generator admitted by a time-fractional partial differential equation and from this result a system was obtained, determining equations, in which a linear equation appears in terms of fractional derivative and another differential equation that depends only on integer derivatives. Furthermore, in this same paper it was concluded that there exists no invertible mapping that converts a nonlinear and linear time-fractional partial differential equation. 

On the other hand, Zhang \cite{zhang2021symmetry} also did another interesting work on Lie symmetry structure of the system consisting of q multi-dimensional time-fractional partial differential equations given by
\begin{eqnarray}\label{90}
    \partial^{\alpha}_{t} {\bf u}= \mathcal{E}(t,{\rm x},{\bf u}),
\end{eqnarray}
where $\mathcal{E}=(\mathcal{E}_{1}\cdots,\mathcal{E}_{q})$ is a smooth vector function involving $p$ independent variables ${\rm x} = (x_{1},\cdots,x_{p})\in\mathbb{R}^{p}$ and $\partial^{\alpha}_{t}(\cdot)$ is the Riemann-Liouville fractional derivative of ${\bf u}= (u_{1},\cdots, u_{q})\in \mathbb{R}^{q}$ ($q$ dependent variables). For more details about the problem (\ref{90}), see \cite{zhang2021symmetry}.

Motivated by the problems (\ref{89}) and (\ref{90}), in this paper our first objective is to consider a new class of differential equations introduced through the Riemann-Liouville fractional derivative with respect to the function $\psi(\cdot)$. In this sense, as a consequence of the formulated problem, our main contributions of this article are best described as follows:
\begin{enumerate}
        \item First, we present the definition and a miscellany of results of the infinitesimal extension for the $\psi$-Riemann-Liouville fractional derivative. Furthermore, we highlight in the paper that this work is the first that involves this type of operator in the $\alpha>0$ case. We highlight here the proof of Zhi-Yong Zhang’s Theorem for $\psi$-Riemann-Liouville fractional derivative. On the other hand, we emphasize that in this extension we consider that a the lower limit of the fractional integral is not fixed, more precisely, we assume that $\overline{a} = a + \epsilon \tau |t=a$ for a
parameter $\epsilon> 0$.

    \item Motivated by the results discussed in the previous item, we carry out some applications, i.e., we consider the Burgers fractional equation and the fractional diffusion equation.
\end{enumerate}

One of the difficulties of working with the $\psi$-Riemann-Liouville fractional derivative operator is trying to control the $\psi(\cdot)$ function by imposing conditions so that the results can be obtained. For example, in the particular case of this work, we cannot choose the function $\psi(t)=log_{a}(t)$ as it does not satisfy the imposed conditions.

In this way, the work is organized as follows: In Section 2  we present the definitions, valid properties for the derivative in the sense of $\psi$-Riemann-Liouville. In the Section 3 we present some important results obtained in the article \cite{zhang2020symmetry} and we apply it to the generalized Burgers fractional equation and by the way, we compared it with the algorithm that was developed in \cite{gazizov2007continuous}. In Section 4 we present the entire theoretical framework for obtaining the generalized extension of the $\psi$-Riemann-Liouville fractional derivative. In the last section we apply the results obtained in some evolution-type fractional equations. Finally, we present considerations.

\section{Mathematical background and main results}

Let $\left[a, b\right]\; \left(0 < a < b < \infty \right)$ be a finite interval on the half-axis $\mathbb{R}^{+}$ and $C\left[a, b\right]$, $AC^{n}\left[a, b \right]$, $C^{n}\left[a, b\right]$ be the spaces of continuous functions, $n$-times absolutely continuous, $n$-times continuously differentiable functions on $\left[a, b\right]$, respectively \cite{sousa2018psi, kilbas1993fractional}.

The space of the continuous function $f$ on $\left[a, b\right]$ with the norm is defined by
\begin{equation}
\lVert f \rVert_{C\left[a,b \right]} =\max_{t \in \left[a,b\right]} |f(t)|.
\end{equation}

The weighted space $C_{\gamma;\psi \left[a,b \right]}$ of functions $f$ on $\left(a\right., b \left. \right]$ is defined by
\begin{equation}
C_{\gamma;\psi }\left[a,b \right]=\left\{f: \left(a\right., b \left. \right] \rightarrow \mathbb{R}; \left(\psi(t)-\psi(a)   \right)^{\gamma} f(t) \in C\left[a,b \right]  \right\}, \; 0 \leq \gamma <1,
\end{equation}
with the norm,
\begin{equation}
\lVert f \rVert_{C_{\gamma;\psi }\left[a,b \right]} =\lVert \left(\psi(t)-\psi(a)   \right)^{\gamma} f \rVert_{}=\max_{t \in \left[a,b\right]} |\left(\psi(t)-\psi(a)   \right)^{\gamma} f(t)|.
\end{equation}

The weighted space $C_{\gamma;\psi }^{n}\left[a,b \right]$ of functions $f$ on $\left(a\right., b \left. \right]$ is defined by
\begin{equation}
C_{\gamma;\psi }^{n}\left[a,b \right]=\left\{f: \left(a\right., b \left. \right] \rightarrow \mathbb{R}; \left(\psi(t)-\psi(a)   \right)^{\gamma} f(t) \in C^{n-1}\left[a,b \right]; f^{(n)}(t) \in  C_{\gamma;\psi}\left[a,b \right]\right\}, \; 0 \leq \gamma <1,
\end{equation}
with the norm
\begin{equation}
\lVert f \rVert_{C_{\gamma;\psi}^{n} \left[a,b \right]} = \sum_{k=0}^{n-1}\lVert  f \rVert_{C\left[a,b \right]}+ \lVert  f^{(n)}\rVert_{\gamma;\psi \left[a,b \right]}=\max_{t \in \left[a,b\right]} |\left(\psi(t)-\psi(a)   \right)^{\gamma} f(t)|,
\end{equation}
where $C^{0}_{\gamma}\left[a,b\right]=C_{\gamma}\left[ a,b \right].$

The weighted space $C_{\gamma;\psi }^{\alpha}\left[a,b \right]$ of functions $f$ on $\left(a\right., b \left. \right]$ is defined by
\begin{equation}
C_{\gamma;\psi}^{\alpha} \left[a,b \right]=\left\{f \in C_{\gamma}\left[ a,b \right]; {}^{RL}_{}\mathcal{D}_{a^{+}}^{\alpha;\psi(t)} f \in C_{\gamma}\left[ a,b \right] \right\}, \; \gamma= \alpha+ \beta(1-\alpha).
\end{equation}

With these spaces exposed, it is possible to list the definitions.

\begin{definition}\label{def1}{\rm \cite{sousa2019leibniz}} Let $\alpha>0$ be a real number $-\infty \leq a < b \leq \infty$, $f$ an integrable function defined on $\left[a, b \right]$ and  $\psi \in C^{1}(\left[a, b\right], \mathbb{R})$ be functions such that $\psi$ is increasing and $\psi'(t)\neq 0$ for all $t \in \left[a, b\right]$. Then, the $\psi$-Riemann-Liouville fractional integral of $f \left( \text{or Riemann-Liouville fractional integral of}\; f\; \text{ with respect to}\; \psi \right)$ of order $\alpha$ $\left(\; \text{left-sided and right-sided}\;  \right)$ is defined, respectively, as
\begin{equation}\label{ipsi}
\mathcal{I}^{\alpha;\psi(t)}_{a^{+}}f(t) = \dfrac{1}{\Gamma(\alpha)} \int_{a}^{t} \psi'(t)(\psi(t)-\psi(s))^{\alpha -1} f(s) ds,
\end{equation}
and
\begin{equation}\label{ipsi2}
\mathcal{I}^{\alpha;\psi(t)}_{b^{-}}f(t) = \dfrac{1}{\Gamma(\alpha)} \int_{t}^{b} \psi'(t) (\psi(t)-\psi(s))^{\alpha -1} f(s) ds,
\end{equation}
and the $\psi$-Riemann-Liouville fractional derivative of $f$ (or Riemann–Liouville fractional derivative of $f$ with respect to $\psi$) of order $\alpha$ $\left(\; \text{left-sided and right-sided}\;  \right)$ is defined,respectively, as
\begin{equation}\label{dpsi}
{}^{RL}_{}\mathcal{D}_{a^{+}}^{\alpha;\psi(t)} f(t) = \dfrac{1}{\Gamma(m-\alpha)}\left(\dfrac{1}{\psi'(t)} \dfrac{d}{dt} \right)^{m} \int_{a}^{t} (\psi(t)-\psi(s))^{m-\alpha -1} \psi'(s)f(s) ds.
\end{equation}
and
\begin{equation}\label{dpsi2}
{}^{RL}_{}\mathcal{D}_{b^{-}}^{\alpha;\psi(t)} f(t) = \dfrac{1}{\Gamma(m-\alpha)}\left(-\dfrac{1}{\psi'(t)} \dfrac{d}{dt} \right)^{m} \int_{t}^{b} (\psi(t)-\psi(s))^{m-\alpha -1} \psi'(s)f(s) ds.
\end{equation}
\end{definition}

Note that, in the Eq.(\ref{ipsi}) and Eq.(\ref{dpsi}), if we take $\psi(t)=t$, yields
\begin{equation}\label{ipsi_class}
\mathcal{I}^{\alpha}_{a^{+}}f(t) = \dfrac{1}{\Gamma(\alpha)} \int_{a}^{t} (t-s))^{\alpha -1} f(s) ds,
\end{equation}
and 
\begin{equation}\label{dpsi_class}
{}^{RL}_{}\mathcal{D}_{a^{+}}^{\alpha} f(t) = \dfrac{1}{\Gamma(m-\alpha)}\left(\dfrac{d}{dt} \right)^{m} \int_{a}^{t} (t-s)^{m-\alpha -1} f(s) ds,
\end{equation}
the classical Riemann-Liouville fractional derivative and integral sense.

\begin{lemma}\label{def2} {\rm \cite{sousa2019leibniz}} Admitting the sets defined above and the conditions for the functions $f$ and $\psi(t)$ in {\bf Definition {\rm \ref{def1}}} we can rewrite the {\rm Eq.(\ref{ipsi})} as:
\begin{equation}\label{int_form2}
\mathcal{I}^{\alpha;\psi(t)}_{a^{+}}f(t)=\sum_{m=0}^{\infty} \binom{-\alpha}{m}f^{[m]}_{\psi}(t)\dfrac{\left[\psi(t)-\psi(a)\right]^{\alpha+m}}{\Gamma(\alpha+m+1)},
\end{equation}
where $t>a$.
\end{lemma}

From the result of the {\bf Lemma \ref{def2}} the proof about the integral of the product of two functions arises as a consequence. As shown below:

\begin{lemma}\label{int_product} {\rm \cite{sousa2019leibniz}} Let  $f$ and $g$ integrable functions  on the interval $\left[a,b \right]$, $\alpha>0$ be and consider a function $\psi \in C^{1}(\left[a,b \right],\mathbb {R})$, such that, it is increasing with $\psi'(t) \neq 0$ for all $t \in \left[a,b \right]$. Then,
\begin{equation}\label{int_product_formula}
\mathcal{I}^{\alpha;\psi(t)}_{a^{+}}(f g)(t):= \displaystyle \sum_{k=0}^{\infty} \binom{-\alpha}{k}  f^{[k]}_{\psi}(t)\; \mathcal{I}^{\alpha+k;\psi(t)}_{a^{+}}g(t).
\end{equation}
\end{lemma}

The {\bf Lemma \ref{int_product}} is a necessary tool to proof Leibniz rule for the $\psi$-Riemann-Liouville fractional derivative. From that, follow the result.

\begin{proposition}[Leibniz rule]\label{LBI} {\rm \cite{sousa2019leibniz}}
Let $\alpha \in \left(n-1,n \right)$ be, $n \in \mathbb{N}$, $f,g$ an integrable functions defined on $\left[a, b\right]$ and  $\psi \in C^{1}(\left[a, b\right], \mathbb{R})$ a function such that $\psi$ is increasing and $\psi'(t)\neq 0$ for all $t \in \left[a, b\right]$. Then, the Leibniz type rule for the $\psi$-Riemann-Liouville fractional derivative is given by
\begin{equation*}
{}^{RL}_{}\mathcal{D}_{a^{+}}^{\alpha;\psi(t)} \left( fg \right)(t) :=\displaystyle \sum_{m=0}^{\infty}\binom{\alpha}{k}f^{[m]}_{\psi}{}^{RL}_{}\mathcal{D}_{a^{+}}^{\alpha -m;\psi(t)} g(t),
\end{equation*}
with
\begin{equation*}
f^{\left[m\right]}_{\psi}:= \left( \dfrac{1}{\psi'(t)} \dfrac{d}{dt}\right)^{m}f(t) .
\end{equation*}
\end{proposition}


\section{Lie symmetry for Riemann-Liouville fractional differential equation}

In this section we present the theory that was elaborated in \cite{gazizov2007continuous} to obtain the infinitesimal prolongation for the Riemann-Liouville fractional derivative. Similar to the case of integer order, the extension of the extension is a \textit{sine qua non} apparatus for obtaining the Lie symmetries of a differential equation. Therefore, it is necessary to build it.

Consider a fractional partial differential equation of the form
\begin{equation}\label{3.1}
\mathcal{\leftindex_{}^{RL}D}_{0^{+}}^{\alpha} u = \mathcal{E}\left[U\right],
\end{equation}
where  $u=u(x,t)$ denotes the unknown function, and
$\mathcal{E}\left[U\right]=\mathcal{E}(x,t,u,u_{x},u_{xx},u_{xt},\ldots)$  is the function that depends on $x$, $t$, and all derivatives of integer order $u$.

In this case $ 0< \alpha \leq 1$ and the derivative considered is classical  Riemann-Liouville fractional derivative Eq.(\ref{dpsi_class}).

From this, we assume the existence of a Lie transformation group, which was demonstrated in detail for the fractional difference equations in \cite{leo2017foundational}. Consequently, it is possible to define the infinitesimal generator for a fractional differential equation in the Riemann-Liouville sense and, therefore, an extended extension in the Riemann-Liouville fractional derivative \cite{gazizov2007continuous}.

Let us assume that Eq.(\ref{3.1}) is invariant under $\epsilon>0$, a continuous transformation parameter. So we can write
\begin{eqnarray}\label{transf}
\bar{t}&=&t+\epsilon \tau(x,t,u)+\mathcal{O}(\epsilon^2), \nonumber\\
\bar{x}&=&x+ \epsilon \xi(x,t,u)+\mathcal{O}(\epsilon^{2})\nonumber\\
\bar{u}&=&u+\epsilon \eta(x,t,u)+\mathcal{O}(\epsilon^{2})\nonumber\\
\mathcal{\leftindex_{}^{RL}D}_{0^{+}}^{\alpha} \overline{u}&=&\mathcal{\leftindex_{}^{RL}D}_{0^{+}}^{\alpha} u+\epsilon \eta^{\alpha}_{t}+\mathcal{O}(\epsilon^{2})\\
\dfrac{\partial \bar{u}}{\partial \bar{x}}&=&\dfrac{\partial u}{\partial x}+ \epsilon \eta_{x}^{(1)}+\mathcal{O}(\epsilon^{2})\nonumber\\
\dfrac{\partial^2 \bar{u}}{\partial \bar{x}^2}&=&\dfrac{ \partial^{2}u }{\partial x^2 }+ \epsilon \eta_{xx}^{(2)}+\mathcal{O}(\epsilon^{2})\nonumber\\
&\vdots\nonumber
\end{eqnarray}
where $\tau(x,t,u)$, $\xi(x,t,u)$, $\eta(x,t,u)$, $\eta_{x}^{(1)}$ e $\eta_{xx}^{(2)}$  are infinitesimal coefficients and  $\eta^{\alpha}_{t}$
is the extended infinitesimal coefficients of order $\alpha$ \cite{gazizov2007continuous,bluman2010applications}.

The infinitesimal generator admitted by Eq.(\ref{3.1}) is given by
\begin{equation}\label{3.3}
X=\tau ( x,t,u)\dfrac{\partial }{\partial t} +\xi ( x,t,u)\dfrac{\partial }{\partial x} \ +\eta ( x,t,u)\dfrac{\partial }{\partial u},
\end{equation}
where the infinitesimals $\tau=\tau (x,t,u), \; \xi=\xi ( x,t,u)$ and $\eta=\eta (x,t,u)$, furthermore,
\begin{equation}
\tau=\dfrac{d \overline{t}}{d \epsilon} \Big|_{\epsilon=0}\;, \xi=\dfrac{d \overline{x}}{d \epsilon} \Big|_{\epsilon=0}\;, \eta=\dfrac{d \overline{u}}{d \epsilon} \Big|_{\epsilon=0}.
\end{equation}
Therefore, to find the Lie point transformation group (\ref{transf}) at which the differential equation becomes invariant is similar to finding the infinitesimal generator Eq.(\ref{3.3}) \cite{leo2017foundational}.

For derivative of order $\alpha$ we have extended infinitesimal generator can be written as
\begin{equation}\label{3.5}
Pr^{(\alpha,l)}X=X +\eta^{\alpha}_{t} \dfrac{\partial}{\partial (\partial_{t}^{\alpha} u)}+ \sum_{i=1}^{l} \eta^{(i)}\dfrac{\partial}{\partial u_{i}},
\end{equation}
where
\begin{align*}
\eta^{(i)}:= D^{i}_{x} \left( \eta -\xi u_{x}-\tau u_{t}\right)+\xi u_{i+1} + \tau u_{it},
\end{align*}
with $u_{it}=\dfrac{\partial^{i+1} u}{\partial x^{i} \partial t}, i=0,1,\cdots,l$ and $\eta^{(0)}=\eta$. Furthermore, 
\begin{equation*}
D_{i}:= \dfrac{\partial}{\partial x^{i}}+u_{i}\dfrac{\partial}{\partial u}+ u_{ij}\dfrac{\partial}{\partial u_{j}}+ \cdots,
\end{equation*}
and
\begin{equation*}
\eta^{\alpha}_{t} = D_{t}^{\alpha}(\eta)+ \xi D_{t}^{\alpha}(u_x)-D_{t}^{\alpha}(\xi u_{x})+D_{t}^{\alpha}(D_{t}(\tau)u)-D_{t}^{\alpha+1}(\tau u)+\tau D_{t}^{\alpha+1}(u),
\end{equation*}
and $l$ is the order of fractional partial differential equation (FPDE) with respect derivative integer order.

From this, using generalized Leibniz rule \cite{gazizov2007continuous,osler1970leibniz}, one has
\begin{align*}
\eta^{\alpha}_{t} &=\dfrac{ \partial^{\alpha}\eta}{\partial t^{\alpha}}+(\eta_{u}-\alpha D_{t}(\tau))\dfrac{\partial^{\alpha} u}{\partial t^{\alpha}}-u\dfrac{\partial^{\alpha}\eta_{u}}{\partial t^{\alpha}}+\mu+
\sum_{n=1}^{\infty}\left[  \binom{\alpha}{n}\dfrac{\partial^{n}\eta_{u}}{\partial t^{n}}-\binom{\alpha}{n+1} D_{t}^{n+1}(\tau)  \right]\times \nonumber\\
&\times D_{t}^{\alpha-n}(u)-\sum_{n=1}^{\infty}\binom{\alpha}{n}D_{t}^{n}(\xi)D_{t}^{\alpha-n}(u_x).
\end{align*}

With the construction carried out above, which was presented for the first time in \cite{gazizov2007continuous} and then used by several other papers \cite{gazizov2011group,leo2014theorem,leo2017foundational,zhang2020symmetry}, it is possible to find Lie symmetries of fractional differential equations.

In the following section we mention a simplification for the infinitesimal generator Eq.(\ref{3.3}), in the case of time-fractional evolution equations, made by observation and demonstration carried out in \cite{zhang2020symmetry} and, consequently, a way that simplifies the system of determining equations that are central to obtaining Lie symmetries .


\section{Lie point symmetry for time-fractional evolution equation}\label{sec4}

In the paper \cite{zhang2020symmetry}, Zhang et al. proved that for fractional evolution equations the infinitesimal generator (\ref{3.3}) can be rewrite as
\begin{equation*}
X=\xi(x) \dfrac{\partial}{\partial x}+ \tau(t) \dfrac{\partial}{\partial t}+\eta(x,t,u) \dfrac{\partial}{\partial u}
\end{equation*}
where $\eta$ is linear in $u$. That is, we can write the infinitesimal generator with the infinitesimals  $\xi(x,t,u)=\xi(x)$, $\tau(x,t,u)=\tau(t)$ and $\eta(x,t,u)=\theta(x)u+\rho(x,t)$.

This was already known \cite{gazizov2007continuous, leo2017foundational}, but it was written in the form of a theorem only \cite{zhang2020symmetry} as can be seen from the following theorem.

Firstly, consider Eq.(\ref{3.1}) in the form 
\begin{equation*}
{\leftindex_{0}^{RL}D}_{t}^{\alpha} u= H\left[U\right]+S(x,t),
\end{equation*}
by the reference \cite{zhang2020symmetry} follow the results.

\begin{theorem}[Z.Y. Zhang's Theorem \cite{zhang2020symmetry}]
If the infinitesimal generator {\rm Eq.(\ref{3.3})} leaves {\rm Eq.(\ref{3.1})} invariant, then $X$ must take the form
\begin{equation*}
X= \xi(x)\dfrac{\partial}{\partial x}+ \tau(t) \dfrac{\partial}{\partial t}+ \eta(x,t,u) \dfrac{\partial}{\partial u},
\end{equation*}
where $\tau(t)\Big|_{t=0}=0$ and $\eta_{uu}=0$.
\end{theorem}

A detailed proof of this theorem can be found in the paper cited above. In addition to the theorem follow the consequence.
\begin{corollary}[\cite{zhang2020symmetry}]\label{corollaz}
In operator $X$, the infinitesimal $\tau(t)$ is expressed explicitly as
\begin{equation}
\tau(t) =c_{2} t^2 + c_{1} t,
\end{equation}
and $\eta(x,t,u)$ is given by
\begin{equation}
\eta = \left\{  \begin{array}{ll}
\theta(x) u + \rho(x,t), \quad c_{2}=0, \\
\dfrac{1}{2} (\alpha -1) (2c_{2}t+ c_{1}) u + \theta(x) u + \rho(x,t), \quad c_{2} \neq 0,
\end{array}
\right.
\end{equation}
where $c_{1}$ and $c_{2}$ are integral constants, $\theta(x)$ and $\rho(x,t)$ are undetermined functions of their arguments, respectively.
\end{corollary}

The consequence of the {\bf Corollary \ref{corollaz}} is that $\tau''=0$, implies $\gamma=\dfrac{\alpha-1}{2}=0$.

One of the important results detailed in the paper that provides a simple method for calculating Lie symmetries for the Riemann-Liouville fractional derivatives is given by the following theorem, however, it is worth mentioning that the symmetries could already be found by another way as presented in \cite{gazizov2007continuous,leo2017foundational,jrtese,soares2022note}, but in Zhang's method the calculations are simpler and makes infinitesimals explicit and, moreover, it has the potential to be implemented computationally.

\begin{theorem}[\cite{zhang2020symmetry}]\label{theoz}
Following the above notations, Lie Symmetries of equations {\rm Eq.(\ref{3.1})} are determined by 
\begin{align}\label{sys1}
\left\{
\begin{array}{ll}
{\leftindex_{}^{RL}D}_{0^{+}}^{\alpha} \rho
+  (\eta_{u} - \alpha \tau^{'})S-\xi S_{x}- \tau S_{t}- \displaystyle\sum_{V} H_{u_{i}} \dfrac{\partial^{i} \rho}{\partial x}=0, \\
(\eta_{u}-\alpha \tau^{'} )H- \xi H_{x}-\tau H_{t} - \displaystyle \sum_{V} H_{u_{i}} (\eta^{(i)}-\dfrac{\partial^{i} \rho}{\partial x}) - \displaystyle \sum_{W/V} H_{u_i} \eta^{(i)}=0
\end{array}
\right.
\end{align}
where the sets $W=\left\{ \text{all terms effective in}\; H  \right\}$,\\
$V=\left\{\; \text{the terms in W which are linear in}\; u_i \right\}$ and
$W\setminus V=\left\{ \text{the  terms contained in}\; W\; \text{but not in}\; V  \right\}$.
\end{theorem}

From these results one can apply the method to find the Lie symmetries for the generalized fractional Burgers' equation (GFBE), for example. Therefore, we have that the result below is a consequence of the application of Zhang's method.

\begin{theorem}[\cite{soares2023note}]
Let the generalized fractional Burgers' equation be
\begin{equation}\label{gfbe}
{\leftindex_{}^{RL}D}_{0^{+}}^{\alpha} u=u_{xx}+g(u)u_{x}, \quad u=u(x,t)
\end{equation}
where $0<\alpha \leq 1$, $g(u)$ smooth function, not constant and Riemann-Liouville fractional derivative sense. Then, the determining equations that give the Lie symmetries is
\begin{align}\label{system}
\left\{
\begin{array}{llll}
{\leftindex_{}^{RL}D}_{0^{+}}^{\alpha} \rho- \rho_{xx}=0, \\
\alpha \tau' - 2\xi'=0, \\
(\theta' u + \rho_{x})g(u) + \theta'' u =0, \\
(\alpha \tau' -\xi' )g(u) + (\gamma \tau' u + \theta u+ \rho)g'(u)-\xi''- 2\theta'=0.
\end{array}
\right.
\end{align}
\end{theorem}

\begin{proof} Consider the generalized fractional Burgers' equation (\ref{gfbe}) by {\bf Theorem \ref{theoz}} and since 
$S(x,t)=0$, $H=u_{xx} +g(u) u_{x}$, $I=\left\{u_{xx}, g(u)u_{x}  \right\}$, $J=\left\{ u_{xx}  \right\}$ and $I/J=\left\{g(u)u_{x} \right\}$ replace all in Eq.(\ref{sys1}), we get
\begin{align}\label{sys}
\left\{
\begin{array}{ll}
{\leftindex_{}^{RL}D}_{0^{+}}^{\alpha} \rho- \displaystyle F_{u_{xx}} \rho_{xx}=0, \\
(\eta_{u}-\alpha \tau^{'} )F - \displaystyle F_{u_{xx}} (\eta^{(2)}_{xx}- \rho_{xx}) - F_{u}\eta -F_{u_{x}} \eta^{(1)}_{x}=0.
\end{array}
\right.
\end{align}

Furthermore, $F_{u}=g'(u)u^{2}_{x},\; F_{u_{x}}=g(u)$ and $F_{u_{xx}}=1$. Then, substituting in the system (\ref{sys}) and separating the equations in terms of the derivatives of $u$ with respect to the independent variable $x$, we obtain
\begin{align}\label{system1}
\left\{
\begin{array}{llll}
{\leftindex_{}^{RL}D}_{0^{+}}^{\alpha} \rho- \rho_{xx}=0, \\
\alpha \tau' - 2\xi'=0, \\
(\theta' u + \rho_{x})g(u) + \theta'' u =0, \\
(\alpha \tau' -\xi' )g(u) + (\gamma \tau' u + \theta u+ \rho)g'(u)-\xi''- 2\theta'=0.
\end{array}
\right.
\end{align}
\end{proof}

\subsection{Gazizov's method for GFBE}

In this subsection we present the algorithm that was established in the article by Gazizov et al. \cite{gazizov2007continuous} in which there is a system of determining equations to find the Lie symmetries of a Riemann-Liouville fractional differential equation.

\begin{theorem}\label{teo90} Let the generalized fractional Burgers' equation be
\begin{equation*}
{\leftindex_{}^{RL}D}_{0^{+}}^{\alpha} u=u_{xx}+g(u)u_{x}, \quad u=u(x,t)
\end{equation*}
where $0<\alpha \leq 1$, $g(u)$ smooth function, not constant and Riemann-Liouville derivative sense. Then, the determining equations is given by follow system:
\begin{equation}\label{system2}
\left\{\begin{array}{lllll}
 \xi_{u}=\xi_{t}=\tau_{u}=\tau_{x}=\eta_{uu}=0,\\
\displaystyle \binom{\alpha}{n} \partial_{t}(\eta_{u})-\displaystyle \binom{\alpha}{n+1}D_{t}^{n+1}(\tau)=0, \; \mbox{to}\; n=1,2,3,\cdots,\\
\xi^{''}(x)-g(u) \alpha \tau^{'}(t)-2\eta_{xu}+g(u)\xi^{'}(x)-\eta g'(u)=0,\\
2\xi^{'}(x)-\alpha \tau^{'}(t)=0,\\
\partial_{t}^{\alpha}(\eta)-u\partial_{t}^{\alpha}\left( \eta_{u}\right)-\eta_{xx}-g(u)\eta_{x}=0.
\end{array}
\right.
\end{equation}
\end{theorem}

\begin{proof} To prove the theorem we will write the following differential function 
\begin{equation*}
\Delta={\leftindex_{}^{RL}D}_{0^{+}}^{\alpha}u -u_{xx}-g(u)u_{x}
\end{equation*}
to apply the algorithm to find the Lie symmetries Riemann-Liouville fractional derivative sense. Using the invariance criterion, its follows that
\begin{equation*}
Pr^{(\alpha,2)}X\Delta=X \Delta+\eta_{x}^{(1)} \dfrac{\partial}{\partial u_{x}}\Delta+\eta_{xx}^{(2)}\dfrac{\partial}{\partial u_{xx}}\Delta+\eta^{\alpha}\dfrac{\partial^{\alpha}}{\partial u^{\alpha}} \Delta=0,
\end{equation*}
when $\Delta =0$.

So, we get
\begin{equation*}
\eta g'(u) u_{x} - \eta_{x}^{(1)}g(u)-\eta_{xx}^{(1)}+\eta^{\alpha}=0.
\end{equation*}

The infinitesimal coefficients $\eta_{x}^{(i)}$ are obtained from the formula given in Eq.(\ref{etainter}), which can be seen in \cite{bluman2010applications}.

This equation depends on the variables $u_{x}, u_{xx}, u_{xt}, u_{t},\cdots$ and $D_{t}^{\alpha-n}u$, $D_{t} ^{\alpha-n}u_{x}$ to $n=1,2,3, \cdots$ which are independent. Substituting the expressions of $\eta_{x}^{(1)}, \eta_{xx}^{(2)},\eta^{\alpha}$ and separating the expressions into powers of $ u$ we get the following system
\begin{equation}
	\left\{\begin{array}{rcl}
	\xi_{u}=\xi_{t}=\tau_{u}=\tau_{x}=\eta_{uu}&=&0  ; \\
		\binom{\alpha}{n} \partial_{t}\left(\eta_{u}\right)-\binom{\alpha}{n+1}D_{t}^{n+1}(\tau)&=&0, \; \mbox{for}\; n=1,2,3,\cdots;   \\
  \xi^{''}(x)- \alpha g(u)  \tau^{'}(t)-2\eta_{xu}+g(u)\xi^{'}(x)-\eta g'(u)&=&0;  \\
  2\xi^{'}(x)-\alpha \tau^{'}(t)&=&0;\\
  \partial_{t}^{\alpha}(\eta)-u\partial_{t}^{\alpha}-\eta_{xx}-g(u)\eta_{x}&=&0.
		\end{array}\right.
\end{equation}
\end{proof}

By the solution of system (\ref{system}) for the case $g(u)$ arbitrary, we obtain that the symmetry admitted by GFBE is $X=\dfrac{\partial}{\partial x}$. In addition to the above symmetry the following symmetries for the particular cases of $g(u)$.

\textbf{Case I: $g(u)=u$}:

In the system (\ref{system}), we obtain $\xi=c_{1} x +c_{2}$ and $\tau=c_{1} \dfrac{2t}{\alpha}$, because $\xi=\xi(x)$ and $\tau=\tau(t)$. In (\ref{system}), $\rho_{x}=0$ and $\theta'=0$.
In (\ref{eq:4}), yields $\gamma=0$ since $\tau$ is linear, $\rho=0$, because the equation is equal to zero and separating the coefficients of $u$, we get $\alpha \tau' -\xi'+\theta=0$. Follow that, $\theta=c_{1}$. Therefore, 
 \begin{equation*}
X_{2}=x \frac{\partial }{\partial x}+\frac{2t}{\alpha}\frac{\partial}{\partial t}-u\frac{\partial}{\partial u}.
 \end{equation*}

\textbf{Case II: $g(u)=u^{p},$ with $p>1$}:

In the system (\ref{system}), becomes
\begin{equation}\label{systemu2}
	\left\{\begin{array}{rcl}
	{}^{RL}D_{0^{+}}^{\alpha} \rho- \rho_{xx}&=&0; \\
		\alpha \tau' - 2\xi'&=&0;   \\
  \theta''u + \left(\theta'u + p \theta + \rho_{x}-\xi'+\alpha \tau'\right)u^{p} &=&0;  \\
  (2\theta' -\xi'' ) + (\gamma \tau' u + \rho)pu^{p-1}&=&0.
		\end{array}\right.
\end{equation}
 
In Eq.(\ref{systemu2}), we obtain $\xi=c_{1}x+c_{2}$ and $\tau=c_{1}\dfrac{2t}{\alpha}$. From the linearity of $\tau$, we have $\gamma=0$ in the Eq.(\ref{systemu2}), in additional, we also obtain $\rho=0$ and $\theta'=0.$ Since this, in Eq.(\ref{systemu2}), $\theta=-\dfrac{c_{1}}{p}$.

Therefore,
\begin{equation*}
X_{2}=x \frac{\partial }{\partial x}+\frac{2t}{\alpha}\frac{\partial}{\partial t}-\frac{u}{p}\frac{\partial}{\partial u}.
\end{equation*}

\textbf{Case III: $g(u)=e^{bu},$ with $b=const.\neq 0$}:
In the system (\ref{system}), yields
\begin{equation}\label{systeme^u}
	\left\{\begin{array}{rcl}
	_{}^{RL}D_{0^{+}}^{\alpha} \rho- \rho_{xx}&=&0; \\
		\alpha \tau' - 2\xi'&=&0;   \\
  (\theta' u + \rho_{x}+\alpha \tau' -\xi' +b\theta u)e^{bu} + \theta'' u  &=&0;  \\
  (\gamma \tau' u + \rho)be^{bu}-\xi''- 2\theta'&=&0.
		\end{array}\right.
\end{equation}

In Eq.(\ref{systeme^u}), we obtain $\xi=c_{1}x+c_{2}$ and $\tau=c_{1}\dfrac{2t}{\alpha}$. From the linearity of $\tau$, we have $\gamma=0$ in the Eq.(\ref{systeme^u}), in additional, we also obtain $\rho=0$ and $\theta'=0$. Since this, in Eq.(\ref{systeme^u}), $\theta u=-\dfrac{c_{1}}{b}.$ From $\eta= \theta u + \rho$, we get
\begin{equation*}
X_{2}=x \frac{\partial }{\partial x}+\frac{2t}{\alpha}\frac{\partial}{\partial t}-\frac{1}{b}\frac{\partial}{\partial u}.
\end{equation*}

\textbf{Case IV: $g(u)=\dfrac{1+u}{u}$}:

In the system (\ref{system}), yields
\begin{equation}\label{4.12}
	\left\{\begin{array}{rcl}
	_{}^{RL}D_{0^{+}}^{\alpha} \rho- \rho_{xx}&=&0; \\
		\alpha \tau' - 2\xi'&=&0;   \\
  -\xi''-\theta' + \left(\theta' + \theta'' \right) u &=&0;  \\
  \left(-\gamma \tau' +\alpha \tau'-\xi' + \rho_{x}-\theta \right) \displaystyle \dfrac{1}{u}-\rho \dfrac{1}{u^2}+\alpha \tau'-\xi'+\rho_{x}&=&0.
		\end{array}\right.
\end{equation}

In Eq.(\ref{4.12}), we obtain $\xi=c_{1}x+c_{2}$ and $\tau=c_{1}\dfrac{2t}{\alpha}$. From the linearity of $\tau$, we have $\gamma=0$ in the equation (\ref{4.12}), in additional, we also obtain $\rho=0$ and $\theta=c_{1}$. Therefore, we get
\begin{equation*}
X_{2}=x\frac{\partial }{\partial x}+ \frac{2t}{\alpha}\frac{\partial}{\partial t }+u  \frac{\partial}{\partial u}.
\end{equation*}
As a particular case of {\bf Theorem \ref{teo90}}, see \cite{soares2022note} in which only \textbf{Case III} was presented.

\begin{remark} We emphasize that both methods for calculating Lie symmetries for fractional derivatives are valid and more than that they represent the same algorithm, however the method introduced by Zhang makes the calculations simpler, due to the fact that it involves a system with just one equation with fractional derivative and the remaining equations of integer order.
\end{remark}

\section{Prolongation for $\psi$-Riemann-Liouville derivative}

In this section we present the definition and results for the infinitesimal extension for the $\psi$-Riemann-Liouville fractional derivative.  It is important to highlight that, in the literature there is no extension for this operator in the case $\alpha>0$. Furthermore, we can see that this operator generalizes the Riemann-Liouville fractional derivative and for $\psi(t)=t$ we recover it. We also emphasize that in this extension considered that $a$ the lower bound of the fractional integral is not fixed, more precisely, we assume that $\overline{a}=a+\epsilon \tau|_{t=a}$ for a parameter $\epsilon >0$.

In other words, we are going to show the  infinitesimal prolongation $\eta^{\alpha; \psi}_{t}$ in the sense of the $\psi$-Riemann-Liouville fractional derivative, i.e., 
\begin{equation}\label{psi-group-eta}
{}^{RL}_{}\mathcal{D}_{\overline{a}^{+}}^{\alpha;\psi(\overline{t})} \overline{u}
={}^{RL}_{}\mathcal{D}_{a^{+}}^{\alpha;\psi(t)} u+\epsilon \eta^{\alpha; \psi(t)}_{t}+\mathcal{O}(\epsilon^{2}).
\end{equation}

For the calculations carried out, it is assumed equations of the type 
\begin{equation}\label{fpde1}
\mathcal{\leftindex_{}^{RL}D}_{a}^{\alpha; \psi(t)} u=\mathcal{E}(x,t,u,u_{x},u_{t}, u_{xx}, u_{xt}, \cdots ),
 \end{equation}
that is, time-fractional partial differential equations,  $u=u(x,t)$ and $\alpha>0$.

In order to facilitate Lie symmetry analysis of Eq.(\ref{fpde1}), we rewrite the equation as
\begin{equation}\label{fpde2}
\mathcal{\leftindex_{}^{RL}D}_{a}^{\alpha; \psi(t)} u=H(x,t,u,u_{x},u_{t}, u_{xx}, u_{xt}, \cdots )+S(x,t),\; \alpha>0,
 \end{equation}
where $H$ has in its argument the dependent variable $u$ or their respective derivatives with respect to $x$ or both of them and the function $S$ is a function of $x$ and $t$ of only the remaining terms in $\mathcal{E}$.

With the purpose to find the Lie symmetries of a fractional partial differential equation in the direction of $\psi$-Riemann-Liouville, for the case of an evolution equation we define the following change of coordinates inspired by the change of coordinates presented in \cite{bluman2010applications, olver2014introduction}.

\begin{definition} Let infinitesimal generator {\rm (\ref{3.3})} be admitted by {\rm Eq.(\ref{fpde2})}, i.e., which is defined in the space-$(x,t,u(x,t ))$. It is possible to denoted another infinitesimal generator admitted by the {\rm Eq.(\ref{fpde2})} such that  
\begin{equation}
X_{\psi}=\xi_{1} (x,\psi(t),\widehat{u})\dfrac{\partial}{\partial x}+\tau_{1}(x,\psi(t),\widehat{u})
\dfrac{\partial}{\partial \psi(t)}+\eta_{1} (x,\psi(t),\widehat{u}) \dfrac{\partial}{\partial \widehat{u}},
\end{equation}
in space-$(x,\psi(t),\widehat{u}),$ where $\widehat{u}=u(x,\psi(t)),$ $\xi_{1}\equiv X(x)=\xi,$ $\tau_{1}\equiv X(\psi(t))=\tau \psi'(t)$ and $\eta_{1}\equiv X(\widehat{u})=\eta$. Therefore, we get
\begin{equation}
X_{\psi}=\xi_{1} (x,\psi(t),\widehat{u})\dfrac{\partial}{\partial x}+ \tau_{1}(x,\psi(t),u)
\dfrac{\partial}{\partial \psi(t)}+\eta_{1} (x,\psi(t),\widehat{u}) \dfrac{\partial}{\partial \widehat{u}}.
\end{equation}
\end{definition}
Note that equality is valid for application in their respective coordinates \cite{bluman2010applications}.

In the paper \cite{gazizov2012fractional} this approach to changing coordinates is presented in which for each function $\psi(t)$ it is possible to rewrite the symmetries. This was inspired by classical theory of the Lie Symmetry \cite{bluman2010applications,olver2014introduction}.

For the $\psi$-Riemann-Liouville fractional derivative the extended infinitesimal generator given in Eq.(\ref{3.5}) can be written as
\begin{equation*}
Pr^{(\left[\alpha;\psi\right],l)}X_{\psi}=X_{\psi} +\eta^{\alpha;\psi}_{t} \dfrac{\partial}{\partial (\partial_{t}^{\alpha;\psi} u)}+ \sum_{i=1}^{l} \eta^{(i)}\dfrac{\partial}{\partial u_{i}},
\end{equation*}
where
\begin{align}\label{etainter}
\eta^{(i)}:= D^{i}_{x} \left( \eta -\xi u_{x}-\tau u_{t}\right)+\xi u_{i+1} + \tau u_{it},
\end{align}
\noindent with
$u_{it}=\dfrac{\partial^{i+1} u}{\partial x^{i} \partial t}, i=0,1,\cdots,l$ and $\eta^{(0)}=\eta$. In addition, we get
$$D_{i}:= \dfrac{\partial}{\partial x^{i}}+u_{i}\dfrac{\partial}{\partial u}+ u_{ij}\dfrac{\partial}{\partial u_{j}}+ \cdots . $$

In consequence of the invariance condition \cite{leo2017foundational}, yields
\begin{equation}\label{invcondgeneral}
Pr^{(\left[\alpha;\psi\right],l)}X \left( \mathcal{\leftindex_{a}^{RL}D}_{t}^{\alpha; \psi} u-H-S\right)\Big|_{\left(\mathcal{\leftindex_{a}^{RL}D}_{t}^{\alpha; \psi} u-H-S=0\right)} =0.
\end{equation}

\begin{remark}
Note that, in general, the published works put an additional condition, which is imposed, $\tau(x,t,u)\Big|_{t=0}=0$, to obtain the invariance, however, we verified that this additional condition does not need to be imposed, since it appears, naturally, in symmetry calculations. As concluded at the end of this work.
\end{remark}

For the case of the $\psi$-Riemann-Liouville fractional derivative, the fractional total derivative with respect to  variable $t$ is given by
\begin{equation}\label{total_def}
    \mathbb{\leftindex_{}^{RL}D}_{a^{+}}^{\alpha; \psi(t)} \left( . \right)=\sum_{m=0}^{\infty}\binom{\alpha}{m} \dfrac{(\psi(t)-\psi(a))^{m-\alpha}}{\Gamma(m+1-\alpha)} D^{m;\psi}_{t}\left( . \right).
\end{equation}

Note that, $\mathbb{D}^{m;\psi(t)}=\left( \dfrac{1}{\psi'(t)}\dfrac{d}{dt}\right)^m \mathbb{D}_{t}^{m} $, where $\mathbb{D}_{t}^{m}$ is total derivative of $m$ order with respect to $t$. It might also be noted, $\mathbb{D}_{t}^{0}(u)=u$ and $\mathbb{D}_{t}^{m+1} u= \mathbb{D}_{t}(\mathbb{D}_{t}^{m} u)$, furthermore $\mathbb{D}_{t}$, for two independent variables is defined as  
\begin{equation*}
\mathbb{D}_{t}= \partial_{t}+ u_{\psi} \partial_{u}+ u_{xt}\partial_{u_{x}}+u_{tt}\partial_{u_{t}}+\cdots.
\end{equation*}

Replacing  $\psi(t)=t$ and $a=0$ in Eq.(\ref{total_def}), it's possible recover to  fractional total derivative of Riemann-Liouville given in \cite{leo2017foundational,zhang2021symmetry}.

\begin{lemma}\label{Le5.2} Let $\alpha>0$ be a real number,  $m \in \mathbb{N}$, with $0<m-\alpha<1$ and $u(x,t)$ a function defined on $C^{n}\left[a, b \right]$ and  $-\infty \leq a < b \leq \infty$,  and  $\psi \in C^{1}(\left[a, b\right], \mathbb{R})$ be functions such that $\psi$ is increasing and $\psi'(t)\neq 0$ for all $t \in \left[a, b\right]$. Consider also, the infinitesimal approximations of order one with respect to the small parameter $\epsilon>0$ of\; $\overline{t}=t+\epsilon \tau(x,t,u)+ \mathcal{O}(\epsilon^{2}),\; \overline{x}=x+\epsilon \xi(x,t,u)+ \mathcal{O}(\epsilon^{2})$ and \; $\overline{u}=u+\epsilon \eta(x,t,u)+ \mathcal{O}(\epsilon^{2})$, then 
\begin{align*}
\left[\psi(\overline{t})-\psi(\overline{a}) \right]^{m-\alpha} \mathbb{D}^{m;\psi(\overline{t})} \overline{u}(\overline{x},\overline{t}) &=
\left[ \psi(t)-\psi(a)  \right]^{m-\alpha} \left[ 1+\epsilon (m-\alpha) \dfrac{\left(  \psi'(t)\tau-\psi'(a)\tilde{\tau}\right)}{\psi(t)-\psi(a)}\right] \times \nonumber \\
&\times \left[ \mathbb{D}^{m;\psi(t)} u(x,t)+\epsilon \eta^{(m;\psi)}\right]+\mathcal{O}(\epsilon^2),
\end{align*}
where 
$\tilde{\tau} =\tau(x,t,u)\Big|_{t=a}$.
\end{lemma}

\begin{proof} Using order one approximations with respect to the parameter $\epsilon>0$, yields
\begin{eqnarray*}
\left[\psi(\overline{t})-\psi(\overline{a}) \right]
&=&\left[\psi(t+\epsilon \tau)-\psi(a+\epsilon \tilde{\tau}) \right]+\mathcal{O}(\epsilon^2)\notag\\
    &=& \psi(t)+\epsilon \psi'(t) \tau-\psi(a)-\epsilon \psi'(a)\tilde{\tau}+\mathcal{O}(\epsilon^2)\notag\\
    &=& \psi(t)-\psi(a)  +\epsilon \left( \psi'(t) \tau-\psi'(a) \tilde{\tau}\right) +\mathcal{O}(\epsilon^2)\notag\\
    &=&\left[\psi(t)-\psi(a) \right] \left( 1+\epsilon \dfrac{\psi'(t)\tau-\tilde{\tau}\psi'(a)}{\psi(t)-\psi(a)} \right)+\mathcal{O}(\epsilon^2).
\end{eqnarray*}

From this, we get
  \begin{align}\label{psi-psia}
&\left[\psi(\overline{t})-\psi(\overline{a}) \right]^{m-\alpha}
    =\left[\left(\psi(t)-\psi(a) \right) \left( 1+\epsilon \dfrac{\psi'(t)\tau-\tilde{\tau}\psi'(a)}{\psi(t)-\psi(a)} \right)\right]^{m-\alpha} +\mathcal{O}(\epsilon^2)\nonumber \\
    &= \left(\psi(t)-\psi(a) \right)^{m-\alpha}\left( 1+\epsilon (m-\alpha)\dfrac{\psi'(t)\tau-\tilde{\tau}\psi'(a)}{\psi(t)-\psi(a)} \right)+\mathcal{O}(\epsilon^2).
    \end{align}

Furthermore, we consider in compliance with Eq.(\ref{transf})
\begin{equation}\label{pgroup-psim}
\mathbb{D}^{m;\psi(\overline{t})} \overline{u}(\overline{x},\overline{t})=\mathcal{D}^{m;\psi(t)} u(x,t)+\epsilon \eta^{(m;\psi)}+\mathcal{O}(\epsilon^2),
\end{equation}
where
\begin{equation}\label{etam}
\eta^{(m;\psi(t))}= \mathbb{D}_{t}^{m,\psi(t)}\left( \eta - \xi u_{x} -\tau u_{t}  \right) + \xi D^{m;\psi(t)}_{t} u_{x} + \tau D^{m+1;\psi(t)}_{t} u.
\end{equation}

From Eq.(\ref{psi-psia}) and Eq.(\ref{pgroup-psim}), we conclude the proof.
\end{proof}

\begin{proposition}\label{prop1} Let {\rm Eq.(\ref{3.3})} be a infinitesimal generator admitted by {\rm Eq.(\ref{3.1})} with $\psi$-Riemann-Liouville fractional derivative, then the $\alpha$-th $\left(\alpha \in \mathbb{R}^{+}\right)$ extended infinitesimal is given by
\begin{equation}\label{lema1}
\displaystyle  \eta^{\alpha; \psi}= {}^{RL}_{}\mathbb{D}_{a}^{\alpha; \psi(t)} \left( \eta - \xi u_{x} -\tau u_{t}\right) + \xi\; {}^{RL}_{}\mathcal{D}_{a}^{\alpha; \psi(t)} u_{x}+ \tau \; \psi'(t)\; {}^{RL}_{}\mathcal{D}_{a}^{\alpha+1; \psi(t)} u +\omega(x,t,u),
\end{equation}
where 
\begin{equation*}
\omega(x,t,u)=\psi'(a)\tilde{\tau} \left( {}^{RL}_{}\mathcal{D}_{a}^{\alpha; \psi(t)} D_{t}^{1,\psi(t)} -  D_{t}^{1,\psi(t)} {}^{RL}_{}\mathcal{D}_{a}^{\alpha; \psi(t)}  \right)u, \;
\text{and}\quad  \tilde{\tau} =\tau(x,t,u)\Big|_{t=a}.
\end{equation*}
\end{proposition}
\begin{proof}

Using the {\bf Proposition \ref{LBI}}, we write
\begin{align*}
\mathbb{\leftindex_{}^{RL}D}_{\overline{a}}^{\alpha; \psi(\overline{t})} \overline{u}=\sum_{m=0}^{\infty} \binom{\alpha}{m} \overline{u}_{\psi}^{[m]} {}^{RL}_{}\mathcal{D}_{\overline{a}}^{\alpha-m; \psi(\overline{t})}(1) 
&=  \sum_{m=0}^{\infty} \binom{\alpha}{m} \dfrac{\left( \psi(\overline{t}) - \psi(\overline{a})   \right)^{m-\alpha}}{\Gamma(m-\alpha+1)} \mathbb{D}_{\overline{t}}^{m;\psi(\overline{t})}\overline{u},
\end{align*}
from Eq.(\ref{psi-group-eta}) we know that 
$\displaystyle  \eta^{\alpha; \psi(t)}= \frac{d}{d\epsilon} \left[
{}^{RL}_{}\mathcal{D}_{\overline{a}}^{\alpha; \psi(\overline{t})}  \left(\overline{u}\right)
\right]_{\epsilon=0}$  and applying {\bf Lemma \ref{Le5.2}}, yields
\begin{equation}\label{eta1}
\displaystyle  \eta^{\alpha; \psi(t)} = \sum_{m=0}^{\infty} \binom{\alpha}{m} \dfrac{(\psi(t)-\psi(a))^{m-\alpha} \eta^{(m;\psi(t))}+(m-\alpha)(\psi(t)-\psi(a))^{m-\alpha-1} \left(\tau \psi'(t)-\tilde{\tau} \psi'(a) \right) \mathbb{D}^{m;\psi(t)}u }{\Gamma(m-\alpha+1)}.
\end{equation}
%

Substituting Eq.(\ref{etam}) in Eq.(\ref{eta1}), one has
\begin{align}\label{eta3}
\displaystyle  \eta^{\alpha; \psi(t)} &= \sum_{m=0}^{\infty} \binom{\alpha}{m} \dfrac{(\psi(t)-\psi(a))^{m-\alpha} \left[  \mathbb{D}^{m;\psi(t)} \left( \eta - \xi u_{x} -\tau u_{t}  \right) + \xi \mathbb{D}^{m;\psi(t)} u_{x} + \tau \mathbb{D}^{m+1;\psi(t)} u \right]  }{\Gamma(m-\alpha+1)} \nonumber\\
&+ \sum_{m=0}^{\infty} \binom{\alpha}{m} \dfrac{(m-\alpha)(\psi(t)-\psi(a))^{m-\alpha-1} \tau \psi'(t) \mathbb{D}^{m;\psi}u}{\Gamma(m-\alpha+1)}- \nonumber\\
&-\sum_{m=0}^{\infty} \binom{\alpha}{m} \dfrac{(m-\alpha)(\psi(t)-\psi(a))^{m-\alpha-1}  \tilde{\tau} \psi'(a) \mathbb{D}^{m;\psi}u}{\Gamma(m-\alpha+1)}.
\end{align}

Using again the {\bf Proposition \ref{LBI}} in Eq.(\ref{eta3}), we obtain
\begin{align}\label{eta4}
\displaystyle  \eta^{\alpha; \psi(t)} &= {}_{}\mathbb{D}_{a}^{\alpha; \psi(t)} \left( \eta - \xi u_{x} -\tau u_{t}\right) + \xi {}_{}\mathbb{D}^{\alpha; \psi(t)} u_{x}+ \sum_{m=0}^{\infty} \binom{\alpha}{m} \dfrac{(\psi(t)-\psi(a))^{m-\alpha} \tau \mathbb{D}^{m+1;\psi(t)}u}{\Gamma(m-\alpha+1)}  \nonumber \\
&+\sum_{m=0}^{\infty} \binom{\alpha}{m} \dfrac{(m-\alpha)(\psi(t)-\psi(a))^{m-\alpha-1} \tau \psi'(t) \mathbb{D}^{m;\psi(t)}u}{\Gamma(m-\alpha+1)}- \nonumber\\
&- \sum_{m=0}^{\infty} \binom{\alpha}{m} \dfrac{(m-\alpha)(\psi(t)-\psi(a))^{m-\alpha-1}  \tilde{\tau} \psi'(a) \mathbb{D}^{m;\psi(t)}u}{\Gamma(m-\alpha+1)} .
\end{align}

Replace $m$ by $m-1$ in the third term and substitute $m=0$ in the last term of Eq.(\ref{eta4})
\begin{align}\label{eta5}
\displaystyle  \eta^{\alpha; \psi(t)} &= {}_{a}\mathbb{D}_{t}^{\alpha; \psi} \left( \eta - \xi u_{x} -\tau u_{t}\right) + \xi {}_{}\mathbb{D}_{t}^{\alpha; \psi(t)} u_{x}+ \sum_{m=1}^{\infty} \binom{\alpha}{m-1} \dfrac{(\psi(t)-\psi(a))^{m-\alpha-1} \tau D_{t}^{m+1;\psi}u}{\Gamma(m-\alpha)} \nonumber \\
&- \dfrac{\alpha (\psi(t)-\psi(a))^{-\alpha-1}\tau u}{\Gamma(1-\alpha)}+\sum_{m=1}^{\infty} \binom{\alpha}{m} \dfrac{(m-\alpha)(\psi(t)-\psi(a))^{m-\alpha-1} \tau \psi'(t) \mathbb{D}_{t}^{m;\psi}u}{\Gamma(m-\alpha)}- \nonumber \\ 
&- \sum_{m=0}^{\infty} \binom{\alpha}{m} \dfrac{(m-\alpha)(\psi(t)-\psi(a))^{m-\alpha-1}  \tilde{\tau} \psi'(a) \mathbb{D}^{m;\psi(t)}u}{\Gamma(m-\alpha+1)} . 
\end{align}

Using the relation  $\displaystyle \binom{\alpha}{m+1}+\binom{\alpha}{m}=\binom{\alpha +1}{m}$ in Eq.(\ref{eta5}), we obtain
 \begin{align*}
\displaystyle  \eta^{\alpha; \psi(t)} &= {}_{}\mathbb{D}^{\alpha; \psi(t)} \left( \eta - \xi u_{x} -\tau u_{t}\right) + \xi {}_{}\mathbb{D}^{\alpha; \psi(t)} u_{x}+ \sum_{m=0}^{\infty} \binom{\alpha+1}{m} \dfrac{(\psi(t)-\psi(a))^{m-\alpha-1} \tau \psi'(t) \mathbb{D}^{m+1;\psi(t)}u}{\Gamma(m-\alpha)}  \nonumber \\ 
&-\sum_{m=0}^{\infty} \binom{\alpha}{m} \dfrac{(m-\alpha)(\psi(t)-\psi(a))^{m-\alpha-1}  \tilde{\tau} \psi'(a) \mathbb{D}_{t}^{m;\psi}u}{\Gamma(m-\alpha+1)}
\\
&= {}_{}\mathbb{D}^{\alpha; \psi(t)} \left( \eta - \xi u_{x} -\tau u_{t}\right) + \xi {}_{}\mathbb{D}^{\alpha; \psi(t)} u_{x}+\tau \psi'(t){}_{}\mathbb{D}^{\alpha+1; \psi(t)} u - \nonumber \\
&\underbrace{-\sum_{m=0}^{\infty} \binom{\alpha}{m} \dfrac{(m-\alpha)(\psi(t)-\psi(a))^{m-\alpha-1}  \tilde{\tau} \psi'(a) \mathbb{D}_{t}^{m;\psi}u}{\Gamma(m-\alpha+1)}}_{(I)} .  \nonumber
\end{align*}

Note that $(I)$ can be write as
\begin{equation*}
\psi'(a) \tilde{\tau} \left[ \left(\frac{1}{\psi'(t)}\frac{d}{d t}\right) {}^{RL}_{a}\mathbb{D}^{\alpha; \psi(t)} u- {}^{RL}_{}\mathbb{D}^{\alpha; \psi} D_{t}^{1,\psi(t)} u \right].
\end{equation*}

In fact,
\begin{align}\label{expres1}
&\sum_{m=0}^{\infty} \binom{\alpha}{m} \dfrac{(m-\alpha)(\psi(t)-\psi(a))^{m-\alpha-1}  \tilde{\tau} \psi'(a) \mathbb{D}_{t}^{m;\psi}u}{\Gamma(m-\alpha+1)}, \nonumber\\
&\sum_{m=0}^{\infty} \binom{\alpha}{m} \tilde{\tau} \psi'(a) \dfrac{  \dfrac{1}{\psi'(t)}\dfrac{\partial}{\partial t}(\psi(t)-\psi(a))^{m-\alpha}   \mathbb{D}_{t}^{m;\psi(t)}u}{\Gamma(m-\alpha+1)} .
\end{align}

Furthermore, 
\begin{align}\label{expres2}
& \left(\dfrac{1}{\psi'(t)}\dfrac{d}{d t} \right) \left[ \left( \psi(t)-\psi(a)  \right)^{m-\alpha} \mathbb{D}^{m;\psi(t)} u  \right] = \nonumber\\
& \left(\dfrac{1}{\psi'(t)}\dfrac{d}{d t} \right)\left( \psi(t)-\psi(a)  \right)^{m-\alpha} \mathbb{D}^{m;\psi} u + \left( \psi(t)-\psi(a)  \right)^{m-\alpha} \mathbb{D}^{m+1;\psi(t)} u .
\end{align}

Replacing Eq.(\ref{expres2}) in Eq.(\ref{expres1}), we get
\begin{align*}\label{expres3}
& \psi'(a) \tilde{\tau} \sum_{m=0}^{\infty} \binom{\alpha}{m} \frac{ \left(\dfrac{1}{\psi'(t)}\dfrac{d}{d t} \right)  \left[ \left( \psi(t)-\psi(a)  \right)^{m-\alpha} \mathbb{D}^{m;\psi(t)} u  \right] - \left( \psi(t)-\psi(a)  \right)^{m-\alpha} \mathbb{D}^{m+1;\psi} u}{\Gamma(m-\alpha+1)}. \nonumber\\
&\tilde{\tau}\left[\left(\dfrac{1}{\psi'(t)}\dfrac{d}{d t} \right) \left( \displaystyle  \sum_{m=0}^{\infty} \binom{\alpha}{m} \dfrac{   \left( \psi(t)-\psi(a)  \right)^{m-\alpha} \mathbb{D}_{t}^{m;\psi} u}{\Gamma(m-\alpha+1)} \right) - \left( \displaystyle\sum_{m=0}^{\infty} \binom{\alpha}{m}\dfrac{  \left( \psi(t)-\psi(a)  \right)^{m-\alpha}  \mathbb{D}^{m+1;\psi(t)} u}{\Gamma(m-\alpha+1)} \right) \right]. \nonumber \\
\end{align*}

Therefore,
\begin{equation*}
\psi'(a) \tilde{\tau} \left[ D^{1,\psi(t)}\;  {}^{RL}_{}\mathcal{D}^{\alpha; \psi(t)} u- {}^{RL}_{}\mathcal{D}^{\alpha; \psi(t)} D^{1,\psi(t)} u \right].
\end{equation*}

Finally, we can write
\begin{align*}\label{final}
\displaystyle  \eta^{\alpha; \psi(t)}= \mathbb{\leftindex_{a}^{RL}D}_{t}^{\alpha; \psi}  \left( \eta - \xi u_{x} -\tau u_{t}\right) + \xi {}_{}\mathcal{D}_{a}^{\alpha; \psi} u_{x}+ \tau \psi'(t) {}_{}\mathcal{D}_{a}^{\alpha+1; \psi} u
+\psi'(a) \tilde{\tau} \left[{}^{RL}_{}\mathcal{D}_{a}^{\alpha; \psi} D_{t}^{1,\psi} u-  D_{t}^{1,\psi} {}^{RL}_{}\mathcal{D}_{a}^{\alpha; \psi} u \right],
\end{align*}
where $\tilde{\tau} =\tau(x,t,u)\Big|_{t=a}$. Therefore we concluded the proof.
\end{proof}

Note that, if $\psi(t)=t$, we get
\begin{align}
\displaystyle
\eta^{\alpha}= 
\mathbb{\leftindex_{a}^{RL}D}_{t}^{\alpha;} \left( \eta - \xi u_{x} -\tau u_{t}\right) 
+ \xi {}_{}\mathcal{D}_{a}^{\alpha} u_{x}+ \tau {}_{}\mathcal{D}_{a}^{\alpha+1} u
+  \tilde{\tau} \left[ {}^{RL}_{}\mathcal{D}_{a}^{\alpha}\; \dfrac{\partial }{\partial t} u - \frac{\partial}{\partial t}\;{}^{RL}_{}\mathcal{D}_{a}^{\alpha} u\right],
\end{align}
where   $\tilde{\tau} =\tau(x,t,u)\Big|_{t=a}$. In this case, we have the extension to the Riemann-Liouville fractional derivative.

If in addition, $a=0$ and $\psi(t)=t$, it is possible to recover the case that has been used frequently in the applications of Lie symmetries for Riemann-Liouville fractional differential equations, as
\begin{align*}
\displaystyle
\eta^{\alpha}= 
\mathbb{\leftindex_{0}^{RL}D}_{t}^{\alpha} \left( \eta - \xi u_{x} -\tau u_{t}\right) 
+ \xi {}_{}\mathcal{D}_{0}^{\alpha} u_{x}+ \tau {}_{}\mathcal{D}_{0}^{\alpha+1} u.
\end{align*}

\begin{lemma} Let  $\mathbb{\leftindex_{a}^{RL}D}_{t}^{\alpha; \psi} $ be the time total derivative in terms of $\psi$-Riemann-Liouville fractional derivative and $\eta(x,t,u)$ infinitesimal, then
\begin{equation}\label{5.21}
\mathbb{\leftindex_{}^{RL}D}_{a}^{\alpha; \psi(t)} \displaystyle  \left(\eta \right)= \mathcal{\leftindex_{}^{RL}D}_{a}^{\alpha; \psi(t)} \displaystyle  \left(\eta \right) + \eta_{u} \mathcal{\leftindex_{}^{RL}D}_{a}^{\alpha; \psi(t)} \displaystyle  \left(u\right)- u\mathcal{\leftindex_{}^{RL}D}_{a}^{\alpha; \psi(t)} \displaystyle  \left(\eta_{u}\right)+ \sum_{m=1}^{\infty} \binom{\alpha}{m} D_{t}^{m,\psi} \left(\eta_{u} \right)\;
\mathcal{\leftindex_{}^{RL}D}_{a}^{\alpha-m; \psi(t)}\left(u\right)+\mu,
\end{equation}
where $\mu$ is given by
\begin{eqnarray}\label{mu}
\mu= \sum_{m=2}^{\infty}\sum_{n=2}^{m}\sum_{k=2}^{n}\sum_{r=0}^{k-1} \binom{\alpha}{m}\binom{m}{n}\binom{k}{r}\dfrac{1}{k!}\dfrac{(\psi(t)-\psi(a))^{m-\alpha}}{\Gamma(m+1-\alpha)}[-u ]^{r}  D_{t}^{m,\psi}\left(u^{k-r} \right) D_{t}^{m-n,\psi} \left(D_{u}^{k,\psi}\eta \right).\notag\\
\end{eqnarray}
\end{lemma}

\begin{theorem} The more detailed expression for the extended prolongation of $\alpha $th order in the $\psi$-Riemann-Liouville fractional derivative sense is given by
\begin{eqnarray}\label{teo1}
 \eta^{\alpha; \psi}&=& \dfrac{\partial^{\alpha;\psi} \eta }{\partial t^{\alpha;\psi}} + \left[  \eta_{u}-\alpha \; {}^{}_{}\mathcal{D}_{t}^{1; \psi}\left(\tau \right)\right] 
 {}^{RL}_{}\mathcal{D}_{a}^{\alpha; \psi(t)} \left(u\right)-u 
 {}^{RL}_{}\mathcal{D}_{a}^{\alpha; \psi(t)} \left(\eta_{u}\right)-\sum_{m=1}^{\infty} \binom{\alpha}{m}
 {}^{}_{}\mathcal{D}_{t}^{m;\psi} \left(\xi \right) 
 {}^{RL}_{}\mathcal{D}_{a}^{\alpha-m; \psi} \left(u_{x}\right)\notag\\
 & +& \sum_{m=1}^{\infty} \left[ \binom{\alpha}{m} {}^{}_{}\mathcal{D}_{t}^{m; \psi} \left( \eta_{u}\right)-\binom{\alpha}{m+1} {}^{}_{}\mathcal{D}_{t}^{m+1; \psi} \left( \tau \right)  \right]{}^{RL}_{}\mathcal{D}_{a}^{\alpha-m; \psi} \left( u\right)+\mu+\omega(x,t,u)\notag\\
\end{eqnarray}
where $\omega(x,t,u)$ and  $\mu$ is given by  {\rm Eq.(\ref{lema1})} and {\rm Eq.(\ref{mu})} respectively.
\end{theorem}

\begin{proof} Using the {\bf Proposition (\ref{LBI})}, we get
\begin{eqnarray}\label{expresionteo}
 {}^{RL}_{}\mathbb{D}_{a}^{\alpha; \psi} \left( \xi u_{x}\right)&=&
 \xi\; {}^{RL}_{}\mathcal{D}_{a}^{\alpha; \psi}\left( u_{x} \right)+
 \sum_{m=1}^{\infty} \binom{\alpha}{m} {}^{}_{}\mathcal{D}_{t}^{m; \psi} (\xi)\; {}^{RL}_{}\mathcal{D}_{a}^{\alpha-m; \psi} \left( u_{x} \right); \nonumber\\
{}^{RL}_{}\mathbb{D}_{a}^{\alpha; \psi(t)} \left(\tau u_{t}\right)&=&\tau\; {}^{RL}_{}\mathcal{D}_{a}^{\alpha+1; \psi}\left( u\right)+ \sum_{m=0}^{\infty} \binom{\alpha}{m+1} {}^{}_{}\mathcal{D}_{t}^{m+1; \psi} \left(\tau\right) {}^{RL}_{}\mathcal{D}_{a}^{\alpha-m; \psi} \left(  u \right);\notag\\
\tau\; {}^{RL}_{}\mathcal{D}_{a}^{\alpha+1; \psi}\left( u\right)&=& \tau\; {}^{RL}_{}\mathcal{D}_{a}^{\alpha+1; \psi}\left(u \right)+ \sum_{m=1}^{\infty}\binom{\alpha+1}{m}  {}^{}_{}\mathcal{D}_{t}^{m; \psi}\left(\tau\right){}^{RL}_{}\mathcal{D}_{a}^{\alpha+1-m; \psi}\left( u \right). \nonumber\\
\end{eqnarray}

Replacing Eq.(\ref{5.21}) and Eq.(\ref{expresionteo}) in Eq.(\ref{lema1}), we have
\begin{eqnarray*}
\displaystyle  \eta^{\alpha; \psi}&=& \dfrac{\partial^{\alpha;\psi} \eta }{\partial t^{\alpha;\psi}} + \left[  \eta_{u}-\alpha \; {}^{}_{}\mathcal{D}_{t}^{1; \psi}\left(\tau \right)\right] 
 {}^{RL}_{}\mathcal{D}_{a}^{\alpha; \psi(t)} \left(u\right)-u 
 {}^{RL}_{}\mathcal{D}_{a}^{\alpha; \psi(t)} \left(\eta_{u}\right)-\sum_{m=1}^{\infty} 
 {}^{}_{}\mathcal{D}_{t}^{m;\psi} \left(\xi \right) 
 {}^{RL}_{}\mathcal{D}_{a}^{\alpha-m; \psi} \left(u_{x}\right)\nonumber\\
 & +& \sum_{m=1}^{\infty} \left[ \binom{\alpha}{m} {}^{}_{}\mathcal{D}_{t}^{m; \psi} \left( \eta_{u}\right)-\binom{\alpha}{m+1} {}^{}_{}\mathcal{D}_{t}^{m+1; \psi} \left( \tau \right)  \right]{}^{RL}_{}\mathcal{D}_{a}^{\alpha-m; \psi} \left( u\right)+\mu+\omega(x,t,u),
\end{eqnarray*}
where 
\begin{equation*}
\mu= \sum_{m=2}^{\infty}\sum_{n=2}^{m}\sum_{k=2}^{n}\sum_{r=0}^{k-1} \binom{\alpha}{m}\binom{m}{n}\binom{k}{r}\dfrac{1}{k!}\dfrac{(\psi(t)-\psi(a))^{m-\alpha}}{\Gamma(m+1-\alpha)}[-u ]^{r}  D_{t}^{m,\psi}\left(u^{k-r} \right) \dfrac{\partial^{m-n+k;\psi} \eta}{\partial t^{m-n;\psi} \partial u^{k}},
\end{equation*}
\begin{equation*}
\omega(x,t,u)=\psi'(a)\tilde{\tau} \left( {}^{RL}_{}\mathcal{D}_{a}^{\alpha; \psi(t)} D_{t}^{1,\psi(t)} -  D_{t}^{1,\psi(t)} {}^{RL}_{}\mathcal{D}_{a}^{\alpha; \psi(t)}  \right)u, \;
\text{and}\quad  \tilde{\tau} =\tau(x,t,u)\Big|_{t=a}.
\end{equation*}
\end{proof}

\begin{lemma}\label{lemma1} It holds that $\mu=0$ if and only if $\eta_{uu}=0$, in other words, $\eta$ is linear with respect $u$.
\end{lemma}

\begin{proof} By the expression of $\mu$ in Eq.(\ref{mu}), if $\eta$ is linear in $u$ then $D_{u}^{k;\psi} \left(\eta \right)=0,$ that implies $\mu=0$. On the other hand, consider the terms $\left(D_{t}^{2,\psi}u\right)^{2}$ which occur uniquely for $n=2$, then we separate the case $n=2$ from $\mu$ and rewrite $\mu$ as the following form
\begin{eqnarray}\label{termsmu}
\mu&=& \sum_{m=2}^{\infty}\sum_{r=0}^{1} \binom{\alpha}{m}\binom{m}{2}\binom{2}{r}\dfrac{1}{2!}\dfrac{(\psi(t)-\psi(a))^{m-\alpha}}{\Gamma(m+1-\alpha)}[-u ]^{r}  D_{t}^{2,\psi}\left(u^{2-r} \right) \dfrac{\partial^{m-2;\psi} }{\partial t^{m-2;\psi}} \left(\eta_{uu}\right) \nonumber\\
&+&\sum_{m=3}^{\infty}\sum_{n=3}^{m}\sum_{k=2}^{n}\sum_{r=0}^{k-1} \binom{\alpha}{m}\binom{m}{n}\binom{k}{r}\dfrac{1}{k!}\dfrac{(\psi(t)-\psi(a))^{m-\alpha}}{\Gamma(m+1-\alpha)}[-u ]^{r}  D_{t}^{m,\psi}\left(u^{k-r} \right) \dfrac{\partial^{m-n+k;\psi} \eta}{\partial t^{m-n+k;\psi} \partial u^{k}} \nonumber\\
&=& \sum_{m=2}^{\infty}\binom{\alpha}{m}\binom{m}{2} \dfrac{(\psi(t)-\psi(a))^{m-\alpha} (D_{t}^{1;\psi} u)^{2} }{\Gamma(m+1-\alpha)}\dfrac{\partial^{m-2;\psi} }{\partial t^{m-2;\psi}} \left(\eta_{uu}\right) \nonumber\\
&+&\sum_{m=3}^{\infty}\sum_{n=3}^{m}\sum_{k=2}^{n}\sum_{r=0}^{k-1} \binom{\alpha}{m}\binom{m}{n}\binom{k}{r}\dfrac{1}{k!}\dfrac{(\psi(t)-\psi(a))^{m-\alpha}}{\Gamma(m+1-\alpha)}[-u ]^{r}  D_{t}^{m,\psi}\left(u^{k-r} \right) \dfrac{\partial^{m-n+k;\psi} \eta}{\partial t^{m-n+k;\psi} \partial u^{k}}.\notag\\
\end{eqnarray}

For the coefficient of $(D_{t}^{1;\psi} u)^{2}$, let $p=m-2$, then it becomes
\begin{eqnarray*}
&&\sum_{p=0}^{\infty}\binom{\alpha}{p+2}\binom{p+2}{2} \dfrac{(\psi(t)-\psi(a))^{p+2-\alpha} }{\Gamma(p+3-\alpha)}\dfrac{\partial^{p;\psi} }{\partial t^{p;\psi}} \left(\eta_{uu}\right)\notag\\&=&\dfrac{1}{2} \alpha (\alpha-1)\sum_{p=0}^{\infty}\binom{\alpha-2}{p} \dfrac{(\psi(t)-\psi(a))^{p+2-\alpha} }{\Gamma(p+3-\alpha)}\dfrac{\partial^{p;\psi} }{\partial t^{p;\psi}} \left(\eta_{uu}\right)\\
&=&\dfrac{1}{2} \alpha (\alpha-1) \mathcal{\leftindex_{a}^{RL}D}_{t}^{\alpha-2; \psi} \left(\eta_{uu}\right).
\end{eqnarray*}

Taking the coefficient of $(D_{t}^{1;\psi} u)^{2}$ to zero by assumption that $\mu$ is equal to zero, we get $\mathcal{\leftindex_{a}^{RL}D}_{t}^{\alpha-2; \psi} \left(\eta_{uu}\right)=0$, then $\eta_{uu}=k(x,u)(\psi(t)-\psi(a))^{\alpha-3}$ with an undetermined function $k(x,t)$. 
It's worth noting that 
\begin{equation}\label{product}
\sum_{r=0}^{k-1}\binom{k}{r}[-u ]^{r}  D_{t}^{m,\psi}\left(u^{k-r} \right) = \sum_{r=0}^{k-1} \left(-1\right)^{r}\dfrac{1}{r!} \sum \dfrac{m!k!}{c_{1}! \cdots c_{m}! c!} u^{k-c} \prod_{j=1}^{n} \left(  \dfrac{D_{t}^{j,\psi} u}{j!}\right)^{c_{j}},
\end{equation}
where $c_{j}$ are non-negative integers and the second sum works on $c_{1}+\cdots+ c_{n}=c \leq k$, $c_{1}+2c_{2}+\cdots+nc_{n}=n$. Then, in Eq.(\ref{product}), the total degree of each term is $k$. In particular for $k=2$, we separate the case of $k=2$ from the last summation of $\mu$ in Eq.(\ref{termsmu}) and rearrange it as the form
\begin{eqnarray}\label{rearranje}
\mu&=&\sum_{m=3}^{\infty} \sum_{n=3}^{m}   \sum_{r=0}^{1} \binom{\alpha}{m}\binom{m}{n}\binom{2}{r}\dfrac{1}{2!}\dfrac{(\psi(t)-\psi(a))^{m-\alpha}}{\Gamma(m+1-\alpha)}[-u ]^{r}  D_{t}^{m;\psi}\left(u^{2-r} \right) \dfrac{\partial^{m-n;\psi} }{\partial t^{m-n;\psi}} \left(\eta_{uu}\right)\nonumber\\
&+&\sum_{m=3}^{\infty}\sum_{n=3}^{m}\sum_{k=3}^{n}\sum_{r=0}^{k-1} \binom{\alpha}{m}\binom{m}{n}\binom{k}{r}\dfrac{1}{k!}\dfrac{(\psi(t)-\psi(a))^{m-\alpha}}{\Gamma(m+1-\alpha)}[-u ]^{r}  D_{t}^{m,\psi}\left(u^{k-r} \right) \dfrac{\partial^{m-n+k;\psi} \eta}{\partial t^{m-n+k;\psi} \partial u^{k}}\notag \\
&=&k(x,u)\sum_{m=3}^{\infty} \sum_{n=3}^{m}  \binom{\alpha}{m}\binom{m}{n} \dfrac{(\psi(t)-\psi(a))^{m-\alpha}}{\Gamma(m+1-\alpha)} D_{t}^{m-n,\psi} \left( \psi(t)-\psi(a) \right)^{\alpha-3} \left[ \sum_{i=1}^{m-1} \binom{m-1}{i}  D_{t}^{i,\psi}u\;  D_{t}^{m-i,\psi} u  \right] \nonumber\\
&+&\sum_{m=3}^{\infty}\sum_{n=3}^{m}\sum_{k=3}^{n}\sum_{r=0}^{k-1} \binom{\alpha}{m}\binom{m}{n}\binom{k}{r}\dfrac{1}{k!}\dfrac{(\psi(t)-\psi(a))^{m-\alpha}}{\Gamma(m+1-\alpha)}[-u ]^{r}  D_{t}^{m,\psi}\left(u^{k-r} \right) \left[  \eta_{u}^{[k]}\right]_{\psi(t)}^{[n-m]},
\nonumber
\end{eqnarray}
where $\left[  \eta_{u}^{[k]}\right]_{\psi(t)}^{[n-m]}=\dfrac{\partial^{m-n;\psi} \eta}{\partial t^{m-n;\psi} \partial u^{k}}\cdot$


By the uniqueness of nonlinear terms  $D_{t}^{i,\psi}u\;  D_{t}^{m-i,\psi} u$ in Eq.(\ref{rearranje}), we obtain their coefficients $k(x,u)=0$ and, then $\eta_{uu}=0$.
\end{proof}

\begin{proposition}\label{pro57} Let {\rm Eq.(\ref{3.3})} be admitted by the fractional evolution equation, then $\mathbf{X_{\psi}}$ is given by
\begin{equation*}
\mathbf{X_{\psi}}= \xi(x)\dfrac{\partial}{\partial x}+  \left( c_{2}(\psi(t)-\psi(a))^2 +c_{1} (\psi(t)-\psi(a))+c_{0}\right) \dfrac{\partial }{\partial \psi}+ \eta(x,\psi(t),u) \dfrac{\partial}{\partial u}, 
\end{equation*}
where 
\begin{equation}\label{etafinal}
\eta = \left\{  \begin{array}{ll}
\theta(x) u + \rho(x,\psi(t)), \quad c_{2}=0, \\
\dfrac{1}{2} (\alpha -1) \left( 2c_{2}(\psi(t)-\psi(a))+ c_{1}\right) u + \theta(x) u + \rho(x,\psi(t)), \quad c_{2} \neq 0.
\end{array}
\right.
\end{equation}
\end{proposition}

\begin{proof} We show the preposition by analyzing the structure of Eq.(\ref{invcondgeneral}) on the solution manifold of Eq.(\ref{fpde2}), expanding condition (\ref{invcondgeneral}) yields
\begin{equation}\label{equationf1}
\displaystyle  \eta^{\alpha; \psi(t)} -\tau(x,\psi(t),u)(H_{\psi(t)}^{[1]}-S_{\psi(t)}^{[1]})-\xi(x,\psi(t),u)(H_{x}-S_{x})- \sum_{i=0}^{l} H_{u_{i}} \eta^{(i)}=0,
\end{equation}
where $\eta^{(i)}$ is given by Eq.(\ref{etainter}) and $\eta^{\alpha; \psi(t)}$ by Eq.(\ref{teo1}), $H_{\psi(t)}^{[1]}=\left(\dfrac{1}{\psi'(t)}\dfrac{d}{dt}\right)H$ and similar for $H_{x}$ and $H_{u_{i}}$ into Eq.(\ref{equationf1}) and vanishing the coefficients of  $\mathcal{\leftindex_{}^{RL}D}^{\alpha-m; \psi(t)}u_{\psi(t)}^{[1]}$, one obtain
\begin{equation}\label{express725}
\binom{\alpha}{m} \mathcal{D}^{m;\psi}_{t} \xi =0,
\end{equation}
which holds for $m=1,2,\cdots$. Then, for $n=1$ we have
$$\mathcal{D}^{1;\psi}_{t} \xi =\xi_{\psi(t)}^{[1]} + u_{\psi(t)}^{[1]}\; \xi_{\psi(u)}^{[1]} =0,$$
which implies $\xi_{t}=\xi_{u}=0,$ i.e. $\xi=\xi(x)$.
Now, looking at the coefficient of $u_{\psi(t)}^{[l-1]}=\left( \dfrac{1}{\psi'(t)}\dfrac{d}{dt}
 \right)\left[\dfrac{\partial^{l} u}{\partial x^{l-1}}\right]$ in Eq.(\ref{equationf1}), which uniquely appears in 
$\eta^{(l)}=D^{(l)}_{x}(\eta-\xi u_{x}-\tau u_{t})+\xi u_{l}+\xi u_{l+1}+\tau u_{lt}$, where
\begin{equation*}
D_{x}^{l}(\tau u_{\psi(t)}^{[1]})=\sum_{j=0}^{l}\binom{l}{j} D^{j}_{x} \tau D^{l-j}_{x}u_{\psi(t)}^{[1]},
\end{equation*}
then the coefficient of $u_{(l-1)t}$ in Eq.(\ref{equationf1}) is  $F_{u_{l}}D_{x}\tau=0$ which means
\begin{equation*}
\mathcal{D}_{x} \tau = \tau_{x} + u_{x}\; \tau_{u} =0,
\end{equation*}
since $F_{u_{l}}\neq 0$. Therefore, implies $\tau_{x}=\tau_{u}=0,$, i.e. $\tau=\tau(t)$. Therefore, Eq.(\ref{equationf1}) becomes
\begin{eqnarray}\label{5.30}
&&\dfrac{\partial^{\alpha;\psi} \eta }{\partial t^{\alpha;\psi}} + \left[  \eta_{u}-\alpha \; {}^{}_{}\mathcal{D}_{t}^{1; \psi}\left(\tau \right)\right] 
 (H-S)-u 
{}^{RL}_{}\mathcal{D}_{a}^{\alpha; \psi(t)} \left(\eta_{u}\right)
 + \sum_{m=1}^{\infty} \left[ \binom{\alpha}{m} {}^{}_{}\mathcal{D}_{t}^{m; \psi} \left( \eta_{u}\right)- \nonumber \right.\\
 &&\left.- \binom{\alpha}{m+1} {}^{}_{}\mathcal{D}_{t}^{m+1; \psi} \left( \tau \right)  \right]{}^{RL}_{}\mathcal{D}_{a}^{\alpha-m; \psi} \left( u\right)+\mu+\omega(x,t,u) -\tau(t)(H_{t}-S_{t})-\xi(x)(H_{x}-S_{x})+  \nonumber\\
 &-& \sum_{i=0}^{l} H_{u_{i}} \eta^{(i)}=0,\notag\\
\end{eqnarray}
where $\eta^{(i)}$ is determined by (\ref{etainter}). Substituting $\eta_{uu}=0$ into Eq.(\ref{5.30}) using {\bf Lemma \ref{lemma1}}, yields
\begin{equation}\label{5.31}
\mathbf{{}^{RL}_{}\mathcal{D}_{a}^{\alpha-m; \psi} \left( u\right):} \binom{\alpha}{m} {}^{}_{}\mathcal{D}_{t}^{m; \psi} \left( \eta_{u}\right)-\binom{\alpha}{m+1} {}^{}_{}\mathcal{D}_{t}^{m+1; \psi} \left( \tau \right)=0,\quad m=1,2, \cdots
\end{equation}
and
\begin{eqnarray}\label{5.32}
&&\dfrac{\partial^{\alpha;\psi} \eta }{\partial t^{\alpha;\psi}} + \left[  \eta_{u}-\alpha \; {}^{}_{}\mathcal{D}_{t}^{1; \psi}\left(\tau \right)\right] 
 (H-S)-u 
{}^{RL}_{}\mathcal{D}_{a}^{\alpha; \psi(t)} \left(\eta_{u}\right)
 +\omega(x,t,u) -\tau(t)(H_{t}-S_{t})+ \nonumber\\
 &-&\xi(x)(H_{x}-S_{x})
 - \sum_{i=0}^{l} H_{u_{i}} \eta^{(i)}=0.\notag\\
\end{eqnarray}

Since $\tau=\tau(t)$, in the Eq.(\ref{5.31}) it is possible to put it in the following format
\begin{equation}\label{5.33}
\left. \begin{aligned} 
 & \binom{\alpha}{1} {}^{}_{}\mathcal{D}_{t}^{1; \psi} \left( \eta_{u}\right)-\binom{\alpha}{2} {}^{}_{}\mathcal{D}_{t}^{2; \psi} \left( \tau \right)=0,\; m=1 \\
  & {}^{}_{}\mathcal{D}_{t}^{m-1; \psi} 
 \left[ {}^{}_{}\mathcal{D}_{t}^{1; \psi} (\eta_{u})-\frac{\alpha-m}{m+1} {}^{}_{}\mathcal{D}_{t}^{2; \psi} \tau  \right]=0,\, m \geq 2. 
\end{aligned} \right\}
\end{equation}

In the first expression in (\ref{5.33}) we write, ${}^{}_{}\mathcal{D}_{t}^{1; \psi} \left(\eta_{u}\right)=\dfrac{\alpha-1}{2} {}^{}_{}\mathcal{D}_{t}^{2; \psi} \left( \tau \right)$ and substituting in second equations gives
\begin{equation}
\left( \dfrac{\alpha-1}{2}-\dfrac{\alpha-m}{m+1}\right) {}^{}_{}\mathcal{D}_{t}^{m-1; \psi} \left(  {}^{}_{}\mathcal{D}_{t}^{2; \psi} \tau\right) =0, \quad m=2,3, \ldots.
\end{equation}

To solve the ordinary differential equations in $t$, we just need to take the first case for $m$, that is, $m=2$. Remembering that we are considering $\tau(t) \Big|_{t=a}=b$, in this case, $a$ and $b$ different from zero. Therefore,
\begin{equation*}
\tau(\psi(t))=c_{0}+c_{1}(\psi(t)-\psi(a))+c_{2}(\psi(t)-\psi(a))^2
\end{equation*}
with three integral constants, $c_{0}, c_{1}$ and $c_{2}$. Then solving the first equation in system (\ref{5.33}) yields two different cases:
\begin{itemize}
\item[(i)] $ {}^{}_{}\mathcal{D}_{t}^{2; \psi} \tau=0$, that is, $c_{2}=0$, then ${}^{}_{}\mathcal{D}_{t}^{1; \psi} \left(\eta_{u}\right)=0$, in additional, we have $\eta_{uu}=0$ in {\bf Lemma \ref{lemma1}}, thus
\begin{equation*}
\tau=b+c_{1}(\psi(t)-\psi(a)),\; \text{and}\; \eta=\theta(x) u +\rho(x,\psi(t))
\end{equation*}
\item[(ii)] ${}^{}_{}\mathcal{D}_{t}^{2; \psi} \tau\neq 0$, that is, $c_{2} \neq 0$, then integrating the first equation in system (\ref{5.33}) in $t$, we get
\begin{equation*}
\eta=\dfrac{1}{2} (\alpha -1) \left( 2c_{2}(\psi(t)-\psi(a))+ c_{1}  \right) u+\theta(x)u +  \rho(x,\psi(t)),
\end{equation*}
\end{itemize}

It is worth mentioning that the functions $\theta(x)$ and $\rho(x,\psi(t))$ are undetermined functions, which can be determined by Eq.(\ref{5.32}).
\end{proof}

\begin{theorem}[Zhi-Yong Zhang's Theorem for $\psi$-Riemann-Liouville, \cite{zhang2020symmetry}] \label{TP}
Following the above notations, Lie Symmetries of fractional evolution equations are determined by 
\begin{align}\label{df6}
\left\{
\begin{array}{ll}
\mathcal{\leftindex_{a}^{RL}D}_{t}^{\alpha;\psi} \rho +(\eta_{u} - \alpha\;  {}^{}_{}\mathcal{D}_{t}^{1; \psi} \tau)G-\xi G_{x}- \tau G_{t}- \displaystyle\sum_{V} H_{u_{i}} \dfrac{\partial^{i} \rho}{\partial x^{i}}+ \omega(x,t,u)=0, \\
(\eta_{u}-\alpha \;{}^{}_{}\mathcal{D}_{t}^{1; \psi}\tau  )H- \xi H_{x}-\tau H_{t} - \displaystyle \sum_{V} H_{u_{i}} (\eta^{(i)}-\dfrac{\partial^{i} \rho}{\partial x^{i}}) - \displaystyle \sum_{W/V} H_{u_i} \eta^{(i)} =0.
\end{array}
\right.
\end{align}
where the sets $W=\left\{\text{all terms effective in}\; H  \right\}$,\\
$V=\left\{\; \text{the terms in W which are linear in}\; u_i \right\}$\; and
$W\backslash V=\left\{\text{the  terms contained in}\; W\; \text{but not in}\; V  \right\}$.
\end{theorem}
\begin{proof}
For the proof of this result, we make use of the result obtained in the {\bf Proposition \ref{pro57}}, in particular, the Eq.(\ref{5.32}) which is given by 
\begin{align}\label{simplify3.1}
&\dfrac{\partial^{\alpha;\psi} \eta }{\partial t^{\alpha;\psi}} + \left[  \eta_{u}-\alpha \; {}^{}_{}\mathcal{D}_{t}^{1; \psi}\left(\tau \right)\right] 
 (H-S)-u 
{}^{RL}_{}\mathcal{D}_{a}^{\alpha; \psi(t)} \left(\eta_{u}\right)
 +\omega(x,t,u) -\tau(t)(H_{\psi}^{[1]}-S_{\psi}^{[1]})+ \nonumber\\
 &-\xi(x)(H_{x}-S_{x})
 - \sum_{i=0}^{l} H_{u_{i}} \eta^{(i)}=0,
\end{align}
where $\tau(\psi(t))=c_{0}+c_{1}(\psi(t)-\psi(a))+c_{2}(\psi(t)-\psi(a))^2$, $\xi(x)$ and $\eta$ is given by Eq.(\ref{etafinal}). Therefore,  Lie Symmetries of Eq.(\ref{fpde2}) are uniquely determined by Eq.(\ref{simplify3.1}). Moreover, on the space $(x,t,u)$, we get
\begin{align}\label{simplify3.2}
\dfrac{\partial^{\alpha;\psi} \eta }{\partial t^{\alpha;\psi}}-u 
{}^{RL}_{}\mathcal{D}_{a}^{\alpha; \psi(t)} \left(\eta_{u}\right)
 =\dfrac{\partial^{\alpha;\psi}  }{\partial t^{\alpha;\psi}} \left(  \eta-u\eta_{u}\right)= \dfrac{\partial^{\alpha;\psi}  }{\partial t^{\alpha;\psi}} \rho(x,t). 
\end{align}
Let $W=\left\{\; \text{all terms in H}\;  \right\}$ and
$V=\left\{\;\text{all terms in H are linear in}\; u_{i}\;  \right\}$. 
Taking terms that have $u$ and derivatives of $u$ with respect to $x$ or not, we write Eq.(\ref{simplify3.1}) in two parts given in (\ref{df6}). Therefore, we complete the proof.

\end{proof}

\section{Some results}
In this section we present the application of the results obtained in the previous sections, in particular, we take the Burgers equations and the linear diffusion equation to exemplify.

The Burgers equation, which is a non-linear equation written in terms of the $\psi$-Riemann-Liouville fractional derivative is given by:
\begin{equation}
\mathcal{\leftindex_{}^{RL}D}_{a^{+}}^{\alpha;\psi(t)}u=g(u)u_{x}+\kappa u_{xx},
\end{equation}
where, $u=u(x,t)$, $\alpha>0\,\text{and} \; \kappa>0.$

\begin{theorem}\label{teobur1}
Let the generalized Burgers' fractional equation be
\begin{equation}\label{gfbe2}
\mathcal{\leftindex_{}^{RL}D}_{a^{+}}^{\alpha;\psi(t)}u=g(u)u_{x}+u_{xx}, \quad u=u(x,t)
\end{equation}
where $0<\alpha<1$, $\kappa=1$ and $g(u)$ smooth function, not constant and $\psi$-Riemann-Liouville fractional derivative. Then, the determining equations that give the Lie symmetries is
\begin{equation}\label{6.3}
\left\{\begin{array}{rcl}
\mathcal{\leftindex_{}^{RL}D}_{a^{+}}^{\alpha;\psi(t)} \rho- \rho_{xx}&=&0 \\
\alpha\;D_{a^{+}}^{1,\psi(t)} \left(\tau\right) - 2\xi'&=&0 \\
 (\theta' u + \rho_{x})g(u) + \theta'' u &=&0\\
(\alpha\;D_{a^{+}}^{1,\psi(t)} \left(\tau\right) -\xi' )g(u) + (\gamma \;D_{a^{+}}^{1,\psi(t)} \left(\tau\right) u + \theta u+ \rho)g'(u)-\xi''- 2\theta'&=&0 \\
 \omega(x,t,u)&=&0.
\end{array}\right.
\end{equation}
\end{theorem}

\begin{proof}
Applying {\bf Theorem \ref{pro57}} and {\bf Theorem \ref{TP}} the system follows.
\end{proof}

From the determining equations given by the system (\ref{6.3}) we obtain the symmetry $X=\dfrac{\partial}{\partial x}$ for arbitrary $g(u)$.

In additional to that symmetry above, there exist for the particular cases, the follow:
\begin{table}[H]
\resizebox{\columnwidth}{!}{%
\begin{tabular}{cclll}
\cline{1-2}
\multicolumn{1}{|c|}{$g(u)$} & \multicolumn{1}{c|}{Additional infinitesimal operators for GFBE} &  &  &  \\ \cline{1-2}
\multicolumn{1}{|c|}{$u$} &
  \multicolumn{1}{c|}{$X_{2}=x \frac{\partial }{\partial x}+\left(\frac{2\left(\psi(t)-\psi(a)\right) }{\alpha}\right)\frac{\partial}{\partial \psi}-u\frac{\partial}{\partial u}$} &
   &
   &
   \\ \cline{1-2}
\multicolumn{1}{|c|}{$u^{p}$  with $p>1$} &
  \multicolumn{1}{c|}{$X_{2}=x \frac{\partial }{\partial x}+\left(\frac{2\left(\psi(t)-\psi(a)\right) }{\alpha}\right)\frac{\partial}{\partial \psi}-\frac{u}{p}\frac{\partial}{\partial u}$} &
   &
   &
   \\ \cline{1-2}
\multicolumn{1}{|c|}{$e^{bu}$ with $b=const. \neq 0$} &
  \multicolumn{1}{c|}{$X_{2}=x \frac{\partial }{\partial x}+\left(\frac{2\left(\psi(t)-\psi(a)\right) }{\alpha}\right)\frac{\partial}{\partial \psi}-\frac{1}{b}\frac{\partial}{\partial u}$} &
   &
   &
   \\ \cline{1-2}

   \multicolumn{1}{|c|}{$\dfrac{u}{1+u}$} &
  \multicolumn{1}{c|}{$X_{2}=x \frac{\partial }{\partial x}+\left(\frac{2\left(\psi(t)-\psi(a)\right) }{\alpha}\right)\frac{\partial}{\partial \psi}+u\frac{\partial}{\partial u}$} &
   &
   &
   \\ \cline{1-2}
\multicolumn{1}{l}{}       & \multicolumn{1}{l}{}                         &  &  & 
\end{tabular}%
}
\caption{Additional Symmetries}
\end{table}

\begin{remark}
Note that, for $\psi(t)=t$, we recover the case to Riemann-Liouville fractional derivative, furthermore, for convenient $\psi(t)$, i.e., monotone increasing and with $\psi'(t) \neq 0$ in the interval in which the derivative is defined, we can obtain other symmetries for each coordinate changes that are represented by several possibilities of $\psi(t)$ functions.

\end{remark}

\begin{remark} Note that, in {\bf Theorem \ref{teobur1}} the equation
$\omega(x,t,u)=0$ by {\bf Proposition \ref{prop1}}, yields
\begin{equation*}
\psi'(a)\tau(a) \left( {}^{RL}_{}\mathcal{D}_{a}^{\alpha; \psi(t)} D_{t}^{1,\psi(t)} -  D_{t}^{1,\psi(t)} {}^{RL}_{}\mathcal{D}_{a}^{\alpha; \psi(t)}  \right)u=0,
\end{equation*}
from the fact that $u(x,t) \neq 0$, it only remains that the product $\psi'(a)\tau(a)=0$, but $\psi'(a) \neq 0$, by definition of the $\psi$-Riemann-Liouville fractional derivative, therefore $\tau(x,t,u)\Big|_{t=a}=0$. This condition is generally imposed in articles that address Lie symmetries of fractional differential equations. However, it is not necessary to impose such a condition, since it emerges naturally in the calculations.
\end{remark}

\begin{theorem}\label{teobur1}
Let the generalized time-fractional diffusion  equation be
\begin{equation}\label{gfbe2}
\mathcal{\leftindex_{}^{RL}D}_{a^{+}}^{\alpha;\psi(t)}u=\left( K(u)u_{x} \right)_{x}, \quad u=u(x,t)
\end{equation}
where $0<\alpha \leq 2$  and $K(u)$ smooth function, not constant and $\psi$-Riemann-Liouville derivative sense. Then, the determining equations that give the Lie symmetries is
\begin{equation}\label{6.1}
\left\{\begin{array}{rcl}
\left(\theta'' u +\rho_{xx} \right)K(u)-\mathcal{\leftindex_{}^{RL}D}_{a^{+}}^{\alpha;\psi(t)} \rho&=&0\\
\left(\gamma\,\;D_{a^{+}}^{1,\psi(t)} \left(\tau\right) u + \theta u+\rho \right)K'(u) + \left(\alpha\,D_{a^{+}}^{1,\psi(t)} \left(\tau\right) - 2\xi' \right)K(u) &=&0\\
 \left(\gamma\,\;D_{a^{+}}^{1,\psi(t)} \left(\tau\right) u + \theta u+\rho \right)K''(u) -\left[\left(\alpha+\gamma \right)\left(D_{a^{+}}^{1,\psi(t)}\left(\tau\right)\right) -2\xi'+\theta  \right]K'(u)&=&0\\
 2 \left(\theta' u + \rho_{x}  \right)K'(u)-\left( \xi'' -2\theta'  \right)K(u)&=&0 \\
 \omega(x,t,u)&=&0.
\end{array}\right.
\end{equation}
\end{theorem}

\begin{proof} Applying {\bf Proposition \ref{pro57}} and {\bf Theorem \ref{TP}} the system follows.
\end{proof}
From the determining equations given by the system (\ref{6.1}), there exist the cases, below:

\textbf{Case $K'(u)=0$}:

The system can be written as,
\begin{equation}\label{6.6a}
\left\{\begin{array}{rcl}
\left(\theta'' u +\rho_{xx} \right)K(u)-\mathcal{\leftindex_{}^{RL}D}_{a^{+}}^{\alpha;\psi(t)} \rho&=&0\\
\left(\alpha\,D_{a^{+}}^{1,\psi(t)} \left(\tau\right) - 2\xi' \right)K(u) &=&0\\-\left( \xi'' -2\theta'  \right)K(u)&=&0\\
 \omega(x,t,u)&=&0.
\end{array}\right.
\end{equation}

Without loss of generality, let $K(u)=1$.

\begin{equation*}
\left\{\begin{array}{rcl}
\mathcal{\leftindex_{}^{RL}D}_{a^{+}}^{\alpha;\psi(t)} \rho&=&\rho_{xx} \\
\theta''u &=&0\\
\alpha\,D_{a^{+}}^{1,\psi(t)} \left(\tau\right) - 2\xi'  &=&0 \\
\xi'' &=&2\theta'\\
\omega(x,t,u)&=&0.
\end{array}\right.
\end{equation*}

We obtain the symmetry $X_{1}=\dfrac{\partial}{\partial x}$, $X_{2}=x\dfrac{\partial}{\partial x}+\dfrac{2}{\alpha} \displaystyle \left( \psi(t)-\psi(a) \right)\dfrac{\partial}{\partial \psi}$, $X_{3}=u \dfrac{\partial}{\partial u}$ and $X_{4}=\rho \dfrac{\partial}{\partial u}$, where $\rho$ is a solution of equation $\mathcal{\leftindex_{}^{RL}D}_{a^{+}}^{\alpha;\psi(t)} \rho=\rho_{xx}.$

\textbf{Case $K'(u)\neq 0$}:

\begin{equation*}
\left\{\begin{array}{rcl}
\left(\theta'' u +\rho_{xx} \right)K(u)-\mathcal{\leftindex_{}^{RL}D}_{a^{+}}^{\alpha;\psi(t)} \rho&=&0 \\
\left(\gamma\,\;D_{a^{+}}^{1,\psi(t)} \left(\tau\right) u + \theta u+\rho \right)K'(u) + \left(\alpha\,D_{a^{+}}^{1,\psi(t)} \left(\tau\right) - 2\xi' \right)K(u) &=&0\\
\left(\gamma\,\;D_{a^{+}}^{1,\psi(t)} \left(\tau\right) u + \theta u+\rho \right)K''(u) -\left[\left(\alpha+\gamma \right)\left(D_{a^{+}}^{1,\psi(t)}\left(\tau\right)\right) -2\xi'+\theta  \right]K'(u)&=&0\\
2 \left(\theta' u + \rho_{x}  \right)K'(u)-\left( \xi'' -2\theta'  \right)K(u)&=&0\\
\omega(x,t,u)&=&0.
\end{array}\right.
\end{equation*}

Separating in terms of derivatives of $K(u)$, we obtain

\begin{equation}\label{4.122}
\left\{\begin{array}{rcl}
\mathcal{\leftindex_{}^{RL}D}_{a^{+}}^{\alpha;\psi(t)} \rho-\omega(x,t,u)&=&0\\
\left(\theta'' u +\rho_{xx}+\alpha\,D_{a^{+}}^{1,\psi(t)} \left(\tau\right) - 2\xi'-\xi'' +2\theta'  \right)K(u)&=&0\\
\left[\left(\gamma\,\;D_{a^{+}}^{1,\psi(t)} \left(\tau\right) + \theta+2\theta'  \right)u +\rho_{x}+\rho+ \right.
-\left. \left(\gamma+\alpha \right)\,\;D_{a^{+}}^{1,\psi(t)} \left(\tau\right)+2\xi'-\theta  \right]K'(u)&=&0\\
\left(\gamma\,\;D_{a^{+}}^{1,\psi(t)} \left(\tau\right)u + \theta u+\rho  \right)K''(u)&=&0.
\end{array}\right.
\end{equation}

Calculating the derivative of all terms in the Eq.(\ref{4.122}) with respect to the variable $x$ and, from the Eq.(\ref{4.122}), we get that $\rho=0$, yields
\begin{equation*}
\left(\theta''' u  - 2\xi''-\xi''' +2\theta''  \right)K(u)=0.
\end{equation*}

\textbf{Subcase $\xi''=0$} and \textbf{$\theta'=0$}:

From this, $\xi=c_{1}x+c_{2}$ and in Eq.(\ref{4.122}), we obtain $\theta=2c_{1}.$ Therefore,
\begin{equation*}
X= x \dfrac{\partial}{\partial x}+2 u \dfrac{\partial}{\partial u}.
\end{equation*}

\textbf{Subcase $\xi'''=0$} and \textbf{$\theta'\neq 0$}:

From this, $\xi=c_{1}x^2 +c_{2}x+c_{3}$ and  combining the equations, we get 
 \begin{equation}\label{etakklinha}
\left( \theta u + \rho \right)'=2 \xi'' \dfrac{K(u)}{K'(u)}.
 \end{equation}

Using Eq.(\ref{etakklinha}) in Eq.(\ref{4.122}), yields
\begin{equation}\label{etakklinha2}
c_{1}\left[ 3 + 4 \left(  \dfrac{K(u)}{K'(u)}\right) \right]=0.
 \end{equation}

If $c_{1}=0$, we obtain $\xi=c_{2}x+c_{3}$ and $\eta=2 \xi' \dfrac{K(u)}{K'(u)}+c_{4}.$

If $c_{1}\neq 0$, we get $\left[ 3 + 4 \left(  \dfrac{K(u)}{K'(u)}\right) \right]=0$, yields 
$K(u)=c_{5}\left(c_{1}+3u  \right)^{-4/3}$. From this,

\begin{equation*}
X_{2}=x^{2} \dfrac{\partial}{\partial x}-x(c_{1}+3u)\dfrac{\partial}{\partial u}.
\end{equation*}

\section{Considerations}

In this work we present the framework to construct Lie Symmetry in the case of derivative of one function in relation to another, which in current literature we call as $\psi$-Riemann-Liouville fractional derivative. It is worth mentioning that in the work \cite{gazizov2012fractional} the extension for $\psi$-Riemann-Liouville fractional derivative was introduced, however for the interval $\left(0,1\right)$, for the fixed lower limit of integration and for equations fractional ordinary differentials. Therefore, in this work the extension to $\alpha \in  \mathbb{R}^{+}$ was presented, the proof of Leibniz's type rule in the case of a fractional derivative of order $\alpha$ in relation to another function, the extension of the method developed in the paper \cite{zhang2020symmetry} for the case of fractional evolution equations with $\psi$-Riemann-Liouville fractional derivative, in addition, we present in Section \ref{sec4} a comparison between Gazizov's method and the method developed by Zhang et al., concluding that both provide the same results, but the method presented in \cite{zhang2020symmetry} is simpler. 

Finally, the application of the theory obtained to find Lie symmetries for the  generalized Burgers' fractional equation and for the nonlinear diffusion equation, time-fractional porous medium equation. In addition to the work presented here, it is possible to think of Lie symmetries for more general fractional operators, such as $\psi$-Caputo and $\psi$-Hadamard, including in variable order. These cases are already under development and will be presented soon.

\section*{AUTHOR DECLARATIONS}

{\bf Conflicts of Interest} The authors have no conflicts to disclose.

\section*{DATA AVAILABILITY}

The data that support the findings of this study are available within the article.

\bibliographystyle{jaacbib}
\bibliography{bibsample}


\end{document}